\documentclass[12pt, hidelinks]{article}
\usepackage{amssymb}

\usepackage[utf8]{inputenc}
\usepackage{tikz}
\usetikzlibrary{matrix,arrows,decorations.pathmorphing}
\usetikzlibrary{decorations.pathreplacing} 
\usetikzlibrary{patterns}
\usepackage{hyperref}
\usepackage[utf8]{inputenc}
\usepackage[english]{babel}
\usepackage{booktabs}
\usepackage{amsmath}
\usepackage{amsthm}
\usepackage{wrapfig}
\usepackage{bm}
\usepackage{mathtools}
\usepackage{caption}
\usepackage{lmodern,wrapfig,amsmath}
\usepackage{enumitem}
\usepackage{ stmaryrd }
\usetikzlibrary{shapes.misc}
\usetikzlibrary{cd}
\usepackage{pgfplots}
\usepackage{tikz-3dplot}
\usepackage{adjustbox}
\tdplotsetmaincoords{60}{115}
\pgfplotsset{compat=newest}
\usepackage{ amssymb }
\usepackage[backend=biber, style=ieee]{biblatex}
\bibliography{bibliography}
\usepackage{centernot}
\usepackage{faktor}
\usepackage{amsmath,amsfonts,amssymb, color, braket}

\newcommand{\adj}{\bigcup_{\mathcal{F}}}
\newcommand{\adjg}{\bigcup_{\mathcal{G}}}

\tikzcdset{scale cd/.style={every label/.append style={scale=#1},
    cells={nodes={scale=#1}}}}

\tikzset{
  curarrow/.style={
  rounded corners=8pt,
  execute at begin to={every node/.style={fill=red}},
    to path={-- ([xshift=50pt]\tikztostart.center)
    |- (#1) node[fill=white] {$\scriptstyle \delta^*$}
    -| ([xshift=-50pt]\tikztotarget.center)
    -- (\tikztotarget)}
    }
}

\newtheorem{theorem}{Theorem}[section]

\newtheorem{definition}[theorem]{Definition}  
\newtheorem{lemma}[theorem]{Lemma}   
\newtheorem{corollary}[theorem]{Corollary}
\newtheorem{remark}[theorem]{Remark}

\title{Vector Bundles over non-Hausdorff Manifolds}
\author{David O'Connell}
\date{\small{\textit{Okinawa Institute of Science and Technology \\
1919-1 Tancha, Onna-son, Okinawa 904-0495, Japan}}} 

\begin{document}

\maketitle

\begin{abstract}
In this paper we generalise the theory of real vector bundles to a certain class of non-Hausdorff manifolds. In particular, it is shown that every vector bundle fibred over these non-Hausdorff manifolds can be constructed as a colimit of standard vector bundles. We then use this description to introduce various formulas that express non-Hausdorff structures in terms of data defined on certain Hausdorff submanifolds. Finally, we use Čech cohomology to classify the real non-Hausdorff line bundles. 
\end{abstract}

\section*{Introduction}

The standard theory of differential geometry assumes that any pair of points in a manifold can be separated by disjoint open sets. This, known as the Hausdorff property, is typically imposed for technical convenience. Indeed, it can be shown that any Hausdorff, locally-Euclidean, second-countable topological space necessarily admits partitions of unity subordinate to any open cover. In turn, these partitions of unity can then be used to construct some familiar features of a manifold. \\

Conversely, non-Hausdorff manifolds will always have open covers that cannot admit partitions of unity \cite{oconnell2023non}. Despite this inconvenience, non-Hausdorff manifolds seem to arise within several areas of mathematics and theoretical physics. Within mathematics, one can find non-Hausdorff manifolds in the leaf spaces of foliations \cite{haefliger1957varietes, gauld2014non}, and in the spectra of certain $C^*$ algebras \cite{deeley2022fell}. Within theoretical physics, some general discussion of non-Hausdorff spacetimes can be found in \cite{penrose1979singularities, visser1995lorentzian, heller2011geometry}, and non-Hausdorff manifolds can be found in the maximal extension of Taub-NUT spacetimes \cite{hajicek1970extensions}, in certain twistor spaces \cite{woodhouse1988geroch}, and to various degrees in the study of branching spacetimes \cite{muller2013generalized, luc2020interpreting, OConnellthesis, hajicek1971causality, earman2008pruning, belnap2021branching}. \\

Despite occurring in both mathematics and physics, a general theory of non-Hausdorff geometry is lacking. In \cite{oconnell2023non}, initial steps were made with a study of the topological properties of non-Hausdorff manifolds. The key idea underpinning their study were the observations of \cite{haefliger1957varietes, OConnellthesis, hajicek1971causality, luc2020interpreting}, which suggest that non-Hausdorff manifolds can always be constructed by gluing together ordinary Hausdorff manifolds along open subspaces. This gluing operation is commonly known as an \textit{adjunction space}, though it may also be seen as the topological colimit of a particular diagram. The goal of this paper is to extend the formalism of \cite{oconnell2023non} to include both smooth manifolds and the vector bundles fibred over them. \\ 

This paper is organised as follows. In Section 1, we start by recalling some details of \cite{oconnell2023non}, and by making some important restrictions on the types of non-Hausdorff manifolds that we will consider. Once this is done, we will show that our non-Hausdorff manifolds naturally inherit smooth structures from the Hausdorff submanifolds that comprise them. We will then describe the ring of smooth real-valued functions, and use a modified version of partitions of unity to extend functions from Hausdorff submanifolds into the ambient non-Hausdorff space. We will show that, due to the contravariance of the $C^\infty$ functor, the space of functions of a non-Hausdorff manifold can be seen as the \textit{fibred product} of the functions defined on each of the Hausdorff submanifolds. \\

In Section 2 we will begin to consider the vector bundles that one may fiber over a non-Hausdorff manifold. In a direct analogy to the theorems of \cite{luc2020interpreting} and \cite{oconnell2023non}, we will show that every vector bundle over a non-Hausdorff manifold can be realised as a colimit of bundles fibred over the Hausdorff submanifolds. We will then use this result to describe the space of sections of a non-Hausdorff vector bundle. A similar fibred-product description will emerge, this time due to the contraviance of the $\Gamma$ functor. This description will then be used to construct Riemmannian metrics on our non-Hausdorff manifolds. \\

In Section 3 we will explore the prospect of classifying the line bundles over a non-Hausdorff manifold. We will do this by appealing to the well-known result that real line bundles over Hausdorff manifolds can be counted by the first Čech cohomology group with $\mathbb{Z}_2$ coefficients. As such, we will first provide descriptions for the Čech cohomology groups of our non-Hausdorff manifolds as Mayer-Vietoris sequences built from the cohomologies of certain Hausdorff submanifolds. We will see that, despite the Čech functor $\check{H}$ being contravariant, in general there is no fibred product description for the cohomology groups of our non-Hausdorff manifolds. This result suggests that a general formula for the exact number of line bundles over a non-Hausdorff manifold is probably not possible. However, in the case that all of the spaces under consideration are connected, we will derive an alternating-sum formula for the number of line bundles that a non-Hausdorff manifold admits. \\

Throughout this paper we will assume familiarity with differential geometry up to the level of \cite{lee2013smooth} or \cite{taubes2011differential}. In Section 3, unless otherwise stated, all results regarding the Čech cohomology are taken from \cite{bott1982differential}. We will assume that all manifolds, Hausdorff or otherwise, are locally-Euclidean, second-countable topological spaces, and as a convention we will use boldface notation to denote non-Hausdorff manifolds and their functions. 

\section{Smooth non-Hausdorff Manifolds}

In this section we will introduce a formalism for smooth non-Hausdorff manifolds. We will start with the relevant topological features, and then we will use these to endow our manifolds with smooth structures. We will then discuss smooth real-valued functions on non-Hausdorff manifolds.   

\subsection{The Topology of non-Hausdorff Manifolds}
We will now briefly review the topological properties of non-Hausdorff manifolds. All facts in this section are stated without proof, since they can already be found in \cite{oconnell2023non}, albeit in a slightly more general form. \\

We begin by generalising the adjunction spaces found in many standard texts such as \cite{hatcher2002algebraic} or \cite{brown2006topology}. Our data $\mathcal{F}$ will consist of three key components: a set $\textsf{M}=\{ M_i \}_{i \in I}$ of $d$-dimensional Hausdorff topological manifolds, a set $\textsf{A} = \{M_{ij}\}_{i,j \in I}$ of bi-indexed submanifolds satisfying $M_{ij} \subseteq M_i$, and a set $\textsf{f} = \{ f_{ij} \}_{i,j \in I}$ of continuous maps of the form $f_{ij}:M_{ij} \rightarrow M_j$. In order to ensure that this data induces a well-defined topological space, we will need to impose some consistency conditions. This is captured in the following definition. 
\begin{definition}\label{DEF: Adjunctive System}
The triple $\mathcal{F} = (\textsf{M}, \textsf{A}, \textsf{f})$ is called an adjunction system if it satisfies the following conditions for all $i,j \in I$.
\begin{enumerate}[itemsep=0.7mm]\setlength{\itemindent}{1em}
    \item[\textbf{A1)}] $M_{ii} = M_i$ and $f_{ii} = id_{M_i}$
    \item[\textbf{A2)}] $M_{ji} = f_{ij}(M_{ij})$, and $f_{ij}^{-1} = f_{ji}$
    \item[\textbf{A3)}]  $f_{ik}(x) = f_{jk} \circ f_{ij}(x)$ for each $x\in M_{ij}\cap M_{ik}$.
\end{enumerate}
\end{definition}
\noindent Given an adjunctive system $\mathcal{F}$, we define the \textit{adjunction space subordinate to} $\mathcal{F}$ to be the quotient of the disjoint union: 
$$ \bigcup_{\mathcal{F}}M_i := \faktor{\left(\bigsqcup_{i\in I} M_i\right)}{\sim}$$
where $\sim$ is the equivalence relation that identifies elements $(x,i)$ and $(y,j)$ of the disjoint union whenever $f_{ij}(x) = y$. Observe that the conditions of Definition \ref{DEF: Adjunctive System} are precisely the conditions needed to ensure that $\sim$ is a well-defined equivalence relation. 
Points in the adjunction space are equivalence classes of elements of each of the spaces $M_i$. We will denote these classes by $[x,i]$, that is, 
$$  [x,i] := \Big\{ (y,j) \in \bigsqcup_{i\in I} M_i \ \big| \ y=f_{ij}(x)   \Big\}.  $$
By construction there exists a collection of continuous maps $\phi_i: M_i \rightarrow \bigcup_{\mathcal{F}} M_i$ that send each $x$ in $M_i$ to its equivalence class $[x,i]$ in the adjunction space. \\

In the case that an adjunction system $\mathcal{F}$ has an indexing set of size $2$ we will refer to $\adj M_i$ as a \textit{binary} adjunction space. It is well-known that these spaces can be equivalently seen as the pushout of the diagram
$$      M_2 \xleftarrow{ \ \ \ f_{12}\ \ \ } M_{12} \xrightarrow { \ \ \ \iota_{12}\ \ \ } M_1 $$
in the category of topological spaces \cite{brown2006topology}. Similarly, it can be shown that general adjunction spaces are colimits of the diagram formed from the data in $\mathcal{F}$, with the maps $\phi_i$ satisfying the analogous universal property.

\begin{lemma}\label{LEM: universal property of adjunction space}
    Let $\mathcal{F}$ be an adjunction system and let $\adj M_i$ be the adjunction space subordinate to $\mathcal{F}$. Suppose that $\varphi_i: M_i \rightarrow X$ is a collection of continuous maps such that $\varphi_i = \varphi_j\circ f_{ij}$ for every $i,j$ in $I$. Then there is a unique continuous map $\alpha: \adj M_i \rightarrow X$ and $\varphi_i = \alpha \circ \phi_i$ for all $i$ in $I$.
\end{lemma}

With this in mind, we will regularly borrow language from category theory and refer to an adjunctive system and its subordinate adjunction space as a \textit{diagram} and its \textit{colimit}, respectively. \\

The canonically-induced maps $\phi_i$ can be particularly well-behaved, provided that we make some extra assumptions on the data in $\mathcal{F}$. The following result makes this precise. 

\begin{lemma}\label{LEM: fij and Aij open implies phi maps open and M gen manifold}
Let $\mathcal{F}$ be an adjunction system in which the gluing regions $M_{ij}$ are all open submanifolds and the maps $f_{ij}$ are open topological embeddings. Then
\begin{enumerate}
    \item the maps $\phi_i$ are all open topological embeddings, and
    \item the adjunction space subordinate to $\mathcal{F}$ is locally-Euclidean and second-countable.
\end{enumerate}
\end{lemma}
Due to the above, we will often refer to the $\phi_i$ as the \textit{canonical embeddings}. We will also use $\textbf{M}, \textbf{N}, ...$ to denote adjunction spaces subordinate to any system satisfying the assumptions of Lemma \ref{LEM: fij and Aij open implies phi maps open and M gen manifold}. Given that each canonical embedding $\phi_i$ acts as a homeomorphism, we will often simplify notation and identify each $M_i$ with its image $\phi_i(M_i)$. \\

According to Lemma \ref{LEM: fij and Aij open implies phi maps open and M gen manifold} the colimit $\textbf{M}$ resulting from a diagram $\mathcal{F}$ of Hausdorff manifolds may be a locally-Euclidean second-countable topological space. However, there is no guarantee that any such $\textbf{M}$ will be non-Hausdorff. As a matter of fact, we will need to assume that the open submanifolds $M_{ij}$ are proper, open submanifolds whose boundaries are pairwise homeomorphic. The following result summarises the consequences of this assumption.

\begin{theorem}\label{THM: Summary of topological properties}
Let $\mathcal{F}$ be an adjunctive system that satisfies the criteria of Lemma \ref{LEM: fij and Aij open implies phi maps open and M gen manifold}, and let $\textbf{M}$ denote the adjunction space subordinate to $\mathcal{F}$. Suppose furthermore that each gluing map $f_{ij}: M_{ij} \rightarrow M_{ji}$ can be extended to a homeomorphism $\overline{f_{ij}}: Cl^{M_{i}}(M_{ij})\rightarrow Cl^{M_{j}}(M_{ji})$ such that each $\overline{f_{ij}}$ satisfies the conditions of Definition \ref{DEF: Adjunctive System}. Then 
\begin{enumerate}
    \item The Hausdorff-violating points in $\textbf{M}$ occur precisely at the $\textbf{M}$-relative boundaries of the subspaces $M_i$.  
    \item Each $M_i$ is a maximal Hausdorff open submanifold of $\textbf{M}$.
    \item If the indexing set $I$ is finite, then $\textbf{M}$ is paracompact.
    \item If each $M_i$ is compact and the indexing set $I$ is finite, then $\textbf{M}$ is compact.
\end{enumerate}
\end{theorem}

Throughout the remainder of this paper we will take $\textbf{M}$ to be a fixed but arbitrary non-Hausdorff manifold that is built as an adjunction of finitely-many Hausdorff manifolds $M_i$ according to both Lemma \ref{LEM: fij and Aij open implies phi maps open and M gen manifold} and Theorem \ref{THM: Summary of topological properties}. Consequently, $\textbf{M}$ is a paracompact, locally-Euclidean second-countable space in which the manifolds $M_i$ sit inside $\textbf{M}$ as maximal Hausdorff open submanifolds.

\subsubsection{Adjunctive Subspaces and Inductive Colimits}
Genreally speaking, given an adjunctive system $\mathcal{F}$, we may form other diagrams by selectively deleting data pertaining to particular indices in the set $I$. This procedure will define what is known as an \textit{adjunctive subsystem}. Given some subset $J \subset I$, and the adjunctive subsystem $\mathcal{F}'$ formed by only considering the data in $J$, we may use Lemma \ref{LEM: universal property of adjunction space} to construct a map between the adjunction spaces subordinate to $\mathcal{F}$ and $\mathcal{F}'$, given by:  
$$\kappa: \bigcup_{\mathcal{F}'} M_i \rightarrow \adj M_i \ \text{where} \ \kappa(\llbracket x,i \rrbracket) = [x,i].$$
Note that here we are using the double-bracket notation to distinguish the two types of equivalence classes. It can be shown that this map $\kappa$ is an open topological embedding whenever $\mathcal{F}$ satisfies the criteria of Lemma \ref{LEM: fij and Aij open implies phi maps open and M gen manifold} \cite[§1.2]{oconnell2023non}.\\

So far we have constructed our non-Hausdorff manifold $\textbf{M}$ by gluing the subspaces $M_i$ together simultaneously. However, it will also be useful to express $\textbf{M}$ as a finite sequence of binary adjunction spaces. We will refer to this sequential construction as an \textit{inductive colimit}. We will now justify the equivalence between these two points of view. The case for a colimit of three manifolds is shown below. 

\begin{lemma}\label{LEM: adj of 3 manifolds is two-step}
Let $\textbf{M}$ be a non-Hausdorff manifold built from three manifolds $M_i$. Then $\textbf{M}$ is homeomorphic to an inductive colimit.
\end{lemma}
\begin{proof}
For readability we will only provide a sketch of the proof here -- the full details can be found in the appendix. We may take the gluing region $A:= M_{13}\cup M_{23}$ of $M_3$, and define the map $f_{13}\cup f_{23}:A \rightarrow M_1 \cup_{f_{12}} M_2$ by sending each element of $A$ to $M_{i3}$ along the gluing map $f_{i3}$, and then into the equivalence class $\llbracket f_{i3} ,i \rrbracket$ in $M_1 \cup_{f_{12}} M_2$. Note that by the gluing condition \textbf{(A3)} of Definition \ref{DEF: Adjunctive System} there is no ambiguity here. We may thus glue $M_3$ to $M_1 \cup_{f_{12}} M_2$ along the map $f_{13}\cup f_{23}$. The universal properties of both $(M_1 \cup_{f_{12}} M_2)\cup_{f_{13} \cup f_{23}} M_3$ and $\textbf{M}$ can then be invoked in order to create the desired homeomorphism.
\end{proof}

As one might expect, we may generalise the previous result to all finite adjunctive systems.

\begin{theorem}\label{THM: inductive colimit}
    Let $\textbf{M}$ be a non-Hausdorff manifold built from $n$-many manifolds $M_i$. Then $\textbf{M}$ is homeomorphic to an inductive colimit. 
\end{theorem}
\begin{proof}
We will proceed by induction on the size of the indexing set $I$. The case of $I=3$ is already proved as the previous result. So, suppose that the hypothesis holds for all non-Hausdorff manifolds with indexing set $I$ of size $n$. Let $\textbf{M}$ be a non-Hausdorff manifold built from a diagram $\mathcal{F}$, with indexing set $I$ of size $n+1$. Consider the subsystem $\mathcal{F}'$ formed from $\mathcal{F}$ by deleting all data pertaining to the space $M_{n+1}$. Then $\mathcal{F}'$ is a well-defined adjunctive system that yields a non-Hausdorff manifold $\textbf{N}$. 

In analogy to Lemma \ref{LEM: adj of 3 manifolds is two-step}, consider the gluing region of $M_{n+1}$ defined by $A:= \bigcup_{i\leq n} M_{(n+1)i}$ together with the map $f: A \rightarrow \textbf{N}$ which sends each element $x$ in $A$ to its equivalence class $\llbracket f_{i (n+1)}(x),i\rrbracket$ in $\textbf{N}$. This yields a binary adjunction space $\textbf{N} \cup_f M_{n+1}$. Again in analogy to Lemma \ref{LEM: adj of 3 manifolds is two-step}, we may invoke the universal properties of both $\textbf{N}\cup_f M_{n+1}$ and $\textbf{M}$ to create a homeomorphism between the two spaces. The result then follows by applying the induction hypothesis to $\textbf{N}$.
\end{proof} 

\subsection{Smooth Structures}
In Lemma \ref{LEM: fij and Aij open implies phi maps open and M gen manifold}, the topological structure of each manifold $M_i$ was preserved by requiring that the gluing maps $f_{ij}$ act as open embeddings. The consequence was that the canonical embeddings $\phi_i: M_i \rightarrow \textbf{M}$ were also open, which ensured that the local Euclidean structure of the manifolds $M_i$ can be transferred to $\textbf{M}$. Formally, this can be achieved by using a collection of atlases $\mathcal{A}_i$ of the manifolds $M_i$ to define an atlas $\mathcal{A}$ on $\textbf{M}$:  
$$ \mathcal{A} := \bigcup_{i \in I} \{ (\phi_i(U_\alpha), \varphi_\alpha \circ \phi_i^{-1}) \ | \ (U_\alpha, \varphi_\alpha) \in \mathcal{A}_i\}.    $$
We will now argue that this technique defines a smooth atlas of $\textbf{M}$, provided that the gluing maps $f_{ij}$ are all smooth.

\begin{lemma}\label{LEM: non-Haus smooth manifold}
Let $\textbf{M}$ be a non-Hausdorff manifold built according to Lemma \ref{LEM: fij and Aij open implies phi maps open and M gen manifold}. If, additionally, the $M_i$ are smooth manifolds and the $f_{ij}$ are all smooth maps, then $\textbf{M}$ admits a smooth atlas.
\end{lemma}
\begin{proof}
Consider the atlas $\mathcal{A}$ of $\textbf{M}$ as described above. We show that the transition maps of this atlas will be smooth in the Euclidean sense. Suppose that $[x,i]$ is a element of $\textbf{M}$, and consider two charts of $\textbf{M}$ at this point. Since $[x,i]$ is an equivalence class, in general these two charts may come from different atlases $\mathcal{A}_i$ and $\mathcal{A}_j$. As such, the charts will be of the form $(\phi_i(U_\alpha), \varphi_\alpha \circ \phi_i^{-1})$ and $(\phi_j(U_\beta), \varphi_\beta \circ \phi_j^{-1})$, where $(U_\alpha,\varphi_\alpha)$ is a chart of $M_i$ at the point $x$, and $(U_\beta,\varphi_\beta)$ is a chart of $M_j$ at the point $f_{ij}(x)$. The transition maps in $\textbf{M}$ will then be: $$  (\varphi_\beta \circ \phi_j^{-1})\circ(\varphi_\alpha \circ \phi_i^{-1})^{-1} = \varphi_\beta \circ (\phi_j^{-1}\circ \phi_i) \circ \varphi_\alpha^{-1} = \varphi_\beta \circ f_{ij} \circ \varphi_\alpha^{-1},   $$
which are smooth since each $f_{ij}$ is. 
\end{proof}

Since each smooth atlas has a unique maximal extension, we may consider the smooth structure induced from the atlas $\mathcal{A}$ described above. This allows us to effectively see our non-Hausdorff manifold $\textbf{M}$ as smooth. We will now introduce some useful criteria for identifying smooth maps.

\begin{lemma}\label{LEM: smooth maps}
Let $\textbf{M}$ and $\textbf{N}$ be smooth (possibly non-Hausdorff) manifolds. A map $f: \textbf{M} \rightarrow \textbf{N}$ is smooth if and only if the restrictions $f|_{M_i}$ are smooth for all $i$ in $I$.
\end{lemma}
\begin{proof}
According to Lemma \ref{LEM: fij and Aij open implies phi maps open and M gen manifold} the collection $M_i$ forms an open cover of $\textbf{M}$. The result then follows as an application of Prop. 2.6 of \cite{lee2013smooth}.
\end{proof}
According to the above result, we may now view the canonical embeddings $\phi_i: M_i \rightarrow \textbf{M}$ as \textit{smooth} open embeddings. We may also argue for a universal property as in Lemma \ref{LEM: universal property of adjunction space} and thus interpret $\textbf{M}$ as the colimit of the diagram $\mathcal{F}$ in the category of smooth locally-Euclidean spaces. The following remark makes precise the primary object of study in this paper. 

\begin{remark}\label{REM: assumptions in this paper}
In addition to the conditions of Lemma \ref{LEM: fij and Aij open implies phi maps open and M gen manifold} and Theorem \ref{THM: Summary of topological properties}, hereafter we will also assume that our non-Hausdorff manifold $\textbf{M}$ is smooth in the sense of Theorem \ref{LEM: non-Haus smooth manifold}.  Moreover, in order to study smooth objects defined on $\textbf{M}$, we will also need to assume that the gluing regions $M_{ij}$ have \textit{diffeomorphic boundaries}, which in this context means that the closures $Cl^{M+i}(M_{ij})$ are all smooth closed submanifolds, and the extended maps $\overline{f_{ij}}$ of Theorem \ref{THM: Summary of topological properties} are all smooth.
\end{remark}

\subsection{Smooth Functions}

Since smoothness is a local property, we may define smooth functions on $\textbf{M}$ as in the Hausdorff case. Moreover, we may still appeal to the structure of the real line to view the space $C^\infty(\textbf{M})$ as a unital associative algebra. In this section we will study $C^\infty(\textbf{M})$ in some detail. To begin with, we will discuss some techniques for constructing functions on $\textbf{M}$. We will then use these techniques to establish a relationship between $C^\infty(\textbf{M})$ and the algebras $C^\infty(M_i)$. 

\subsubsection{The Construction of Functions on $\textbf{M}$}

In the Hausdorff setting, there are two useful techniques for constructing smooth functions on a manifold. Roughly speaking, these are:
\begin{enumerate}
    \item to glue together smooth functions defined on open subsets, and 
    \item to extend functions defined on a closed subset.
\end{enumerate}
Will will refer to these two techniques are often known as the ``Gluing Lemma" and the ``Extension Lemma", respectively. For a more thorough discussion of these two constructions, the reader is encouraged to see Corollary 2.8 and Lemma 2.26 of \cite{lee2013smooth}. \\

We will now create versions of these two results for our non-Hausdorff manifold $\textbf{M}$. To begin with, we show that the smooth functions on $\textbf{M}$ can be built by gluing together a collection of smooth functions that are defined on the Hausdorff submanifolds $M_i$.
\begin{lemma}\label{LEM: Adjunction of a smooth function}
If $r_i: M_i \rightarrow \mathbb{R}$ is a collection of smooth functions such that $r_i = r_j\circ f_{ij}$ for all $i,j$ in $I$, then the map $\textbf{r}: \textbf{M}\rightarrow \mathbb{R}$ defined by $\textbf{r}([x,i]) = r_i(x)$ is a smooth function on $\textbf{M}$.
\end{lemma}
\begin{proof}
Observe first that $\textbf{r}$ is well-defined -- if we have $[x,i]=[y,j]$, then $x\in M_{ij}$ with $f_{ij}(x)=y$. So, we have that $$\textbf{r}[y,j] = r_j(y) = r_j(f_{ij}(x)) =r_i(x) = \textbf{r}[x,i]. $$
Moreover, the restriction of $\textbf{r}$ to each $M_i$ equals $r_i$, which is smooth by assumption. The result then follows from an application of Lemma \ref{LEM: smooth maps}. 
\end{proof}

The above result is a straightforward  analogue of the Gluing Lemma. In contrast, an analogue of the Extension Lemma is more involved. The underlying complication is that the extension of a functions defined on a closed subset necessarily requires partitions of unity subordinate to any open cover. As proved in \cite{oconnell2023non}, in the non-Hausdorff setting we have the following obstruction to the existence of partitions of unity.
\begin{lemma}\label{LEM: no partition of unity for Hausdorff cover}
Any open cover of $\textbf{M}$ by Hausdorff sets does not admit a partition of unity subordinate to it.
\end{lemma}
Since each $M_i$ is an open, Hausdorff submanifold of $\textbf{M}$, we cannot directly use any partitions of unity subordinate to the cover $\{M_i\}$. However, our restrictions on the topology of $\textbf{M}$ are stringent enough so as to allow certain techniques involving partitions of unity. Indeed, the requirement that the gluing regions $M_{ij}$ have diffeomorphic boundaries may allow us to smoothly transfer objects between the submanifolds $M_i$, and the requirement that $\textbf{M}$ be a finite colimit may allow this transfer to be performed inductively. We now illustrate this approach with a construction of non-zero functions on $\textbf{M}$.

\begin{theorem}\label{THM: functions non-empty}
Any function $r_i$ on $M_i$ can be extended to a function on $\textbf{M}$. 
\end{theorem}
\begin{proof}
We proceed by induction on the size of $I$. Suppose first that $\textbf{M}$ is a binary adjunction space $M_1 \cup_{f_{12}} M_2$, and without loss generality suppose that $i=2$. Let $r_2$ be any smooth function on $M_2$. The restriction of $r_2$ to the closed submanifold $Cl^{M_2}(M_{12})$ is also a smooth function, and moreover the composition $$ r_2 \circ \overline{f_{12}}: Cl^{M_1}(M_{12}) \rightarrow \mathbb{R} $$ is a smooth function on the copy of $Cl^{M_1}(M_{12})$ that sits inside $M_1$. We can now use a partition of unity argument on $M_1$ to extend $r_2 \circ \overline{f}_{12}$ to a function $r_1$ defined on all of $M_1$. We have thus created a pair of functions $r_i$ in $C^\infty(M_i)$ that agree on $M_{12}$. This pair of functions satisfies the antecedent of Lemma \ref{LEM: Adjunction of a smooth function}, and thus define a function $\textbf{r}$ on $\textbf{M}$ which restricts to $r_2$ on $M_2$. \\

Suppose now the hypothesis holds for all non-Hausdorff manifolds constructed as the colimit of $n$-many Hausdorff manifolds $M_i$, according to Theorem \ref{THM: Summary of topological properties}. Let $\textbf{M}$ be a non-Hausdorff manifold defined as the colimit of $(n+1)$-many manifolds $M_i$. Without loss of generality, pick any smooth function $r_n$ defined of $M_n$. According to Theorem \ref{THM: inductive colimit}, we may view $\textbf{M}$ as the inductive colimit: $$ \textbf{N} \cup_f M_{n+1},$$ where we glue along the set $A:= \bigcup_{i\leq n}  M_{(n+1)i}$. By the induction hypothesis, there exists some non-zero function $\textbf{r}$ defined on the adjunction space $\textbf{N}$ that extends $r_n$. 

Using the fact that each $f_{ij}$ can be extended to a diffeomorphism of boundaries, we can extend the collective function $f$ to a closed function $\overline{f}: Cl^{M_{n+1}}(A)\rightarrow Cl^{\textbf{N}}(f(A))$. The map $\overline{f}$ is well defined since the extensions $\overline{f_{(n+1)i}}$ satisfy a cocycle condition as in Definition \ref{DEF: Adjunctive System}, and moreover $\overline{f}$ is a diffeomorphism since locally it equals $\overline{f_{(n+1)i}}$.

We can restrict $\textbf{r}$ to $Cl^{\textbf{N}}(A)$ and then the map $\textbf{r} \circ \overline{f}$ will be a smooth function defined on $Cl^{M_{n+1}}(A)$. Using a partition of unity on $M_{n+1}$, we may extend the function $\textbf{r} \circ \overline{f}$ to some function $r'$ defined on all of $M_{n+1}$. We may then use Lemma \ref{LEM: Adjunction of a smooth function} on $\textbf{r}$ and $r'$ to form a globally-defined function on all of $\textbf{M}$, which by construction will restrict to $r_n$ on $M_n$.  
\end{proof}

Usefully, smooth functions can be pulled back along smooth maps via precomposition. In the case of the canonical embeddings $\phi_i$, precomposition gives an algebra morphism $\phi_i^*: C^\infty(\textbf{M})\rightarrow C^\infty(M_i)$. As an immediate application of Theorem \ref{THM: functions non-empty} we make the following observation. 

\begin{corollary}\label{COR: phii* map is surjective}
For each $M_i$, the map $\phi^*_i: C^\infty(\textbf{M})\rightarrow C^\infty(M_i)$ is surjective.
\end{corollary}

\subsubsection{The Fibre Product Structure of $C^\infty(\textbf{M})$}
We saw in the previous section that we can always create smooth functions on $\textbf{M}$ by gluing together functions that are defined on the component spaces $M_i$, provided that they are compatible on the overlaps $M_{ij}$. The following result expresses this principle at the level of algebras. 

\begin{theorem}\label{THM: fibred product of smooth functions}
The algebra $C^{\infty}(\textbf{M})$ is isomorphic to the fibred product $$ \prod_{\mathcal{F}} C^\infty(M_i) :=\big\{ (r_1, ..., r_n) \in \bigoplus_{i\in I}  C^\infty(M_i) \big{|} \ r_i = r_j \circ f_{ij} \ on \ M_{ij} \ \text{for all } i,j\in I \big\}.      $$
\end{theorem}
\begin{proof}
Consider the map $\Phi^*$ that acts on each smooth function $\textbf{r}$ on $\textbf{M}$ by: $$  \Phi^*(\textbf{r}) := (\phi_1^*r,\cdots, \phi_n^*r).           $$
By the commutativity of the diagram $\mathcal{F}$, we have that $$ \phi_j^*\textbf{r} \circ f_{ij} (x) = \textbf{r}([f_{ij}(x),j]) = \textbf{r}([x,i]) = \phi_i^*\textbf{r}(x), $$ thus $\Phi^*$ takes image in the fibred product $\prod_{\mathcal{F}} C^\infty(M_i)$.
Moreover, the map $\Phi^*$ is also an algebra homomorphism, since all the $\phi_i^*$ are. The map $\Phi^*$ is injective since any pair of distinct functions $\textbf{r}$ and $\textbf{r}'$ on $\textbf{M}$ must differ on one of the $M_i$, and it is surjective by Corollary \ref{COR: phii* map is surjective}. Since every bijective algebra homomorphism is an isomorphism, this completes the proof.
\end{proof}

We saw in the form of Lemma \ref{LEM: universal property of adjunction space} that the adjunction space $\textbf{M}$ is the colimit $\mathcal{F}$ in the category of topological spaces. Subsequent remarks in Section 1.2 confirmed that this colimit also exists in the category of smooth locally-Euclidean manifolds. Since $C^\infty$ is a contravariant functor, in principle we may apply it to all of $\mathcal{F}$ to obtain a diagram in the category of unital associative algebras. The following result confirms that $C^\infty(\textbf{M})$ is the correct limit of this contravariant diagram. 
\begin{lemma}\label{LEM: universal property of Cinf functions}
Let $A$ be a unital associative algebra together with a collection of $I$-many algebra morphisms $\rho_i: A \rightarrow C^\infty(M_i)$. If the maps $\rho_i$ satisfy $\rho_i = \rho_j \circ f_{ij}$ for all $i,j$ in $I$, then there is a unique algebra morphism $\alpha: A \rightarrow \prod_{\mathcal{F}} C^\infty(M_i)$. 
\end{lemma}
\begin{proof}
Consider the map $\alpha$ defined by $$ \alpha(a) = (\rho_1(a),...,\rho_n(a)). $$ This is clearly an element of the direct sum $\bigoplus_i C^\infty(M_i)$, and moreover the commutativity assumption of the $\rho_i$ ensures that $\alpha$ takes image in the fibred product. That $\alpha$ is an algebra morphism follows from the fact that each $\rho_i$ is. Finally, the uniqueness of $\alpha$ is guaranteed since the morphisms from $\prod_{\mathcal{F}} C^\infty(M_i)$ to $C^\infty(M_i)$ are all projections.
\end{proof}

\section{Vector Bundles}
In this section we will construct vector bundles over our non-Hausdorff manifold $\textbf{M}$. We will start by defining a an adjunction of Hausdorff bundles $E_i$ that are fibred over each of the submanifolds $M_i$. After this, we will argue that every vector bundle over $\textbf{M}$ can be constructed in this manner. We then will generalise Theorem \ref{THM: fibred product of smooth functions} by providing a description of sections of any vector bundle fibred over $\textbf{M}$, eventually finishing with a discussion of Riemannian metrics in the non-Hausdorff setting. All of the basic details of vector bundles can be found in standard texts such as \cite{lee2013smooth, taubes2011differential}.

\subsection{Colimits of Bundles}
Suppose that we are given a collection of rank-$k$ vector bundles $E_i\xrightarrow{\pi_i} M_i$. Since each $M_{ij}$ is an open submanifold of $M_i$, we can always form the restricted bundle $E_{ij}:= \iota_{ij}^*E_i$, which will be an open submanifold of $E_i$. In order to form an adjunction space from this data we need a collection of bundle morphisms $F_{ij}:E_{ij} \rightarrow E_j$ that cover the gluing maps $f_{ij}$. According to Definition \ref{DEF: Adjunctive System}, these maps also need to satisfy the cocycle condition $ F_{jk} \circ F_{ik} = F_{ij}$ on the triple intersections $E_{ij} \cap E_{ik}$. \\

With all of this data in hand, we may use Lemma \ref{LEM: non-Haus smooth manifold} on the collection of  bundles $E_i \xrightarrow{\pi_i} M_i$ to form a non-Hausdorff smooth manifold $\textbf{E}$. We denote by $\chi_i$ the canonical embeddings of each $E_i$ into $\textbf{E}$. In order to describe a bundle structure on $\textbf{E}$, we would first like to define a projection map $\boldsymbol{\pi}:\textbf{E} \rightarrow \textbf{M}$. This amounts to completing the commutative diagram
\begin{center}
    \begin{tikzcd}[row sep=1.9em, column sep=1.8em]
 & E_{ij} \arrow[ld, "F_{ij}"'] \arrow[dd, "\pi_{ij}"', near start] \arrow[rr, "I_{ij}"] &  & E_i \arrow[dd, "\pi_i"] \arrow[ld, "\chi_i"', dashed] \\
E_j \arrow[rr, "\chi_j", near end, crossing over, dashed] \arrow[dd, "\pi_j"'] &  & \textbf{E}  &  \\
 & M_{ij} \arrow[ld, "f_{ij}"', near start] \arrow[rr, "\iota_{ij}", near start] &  & M_i \arrow[ld, "\phi_i"] \\
M_j \arrow[rr, "\phi_j"'] &  & \textbf{M} \arrow[from=uu, "\boldsymbol{\pi}", dashed, near start, crossing over] & 
\end{tikzcd}
\end{center}
simultaneously for all $i,j$ in $I$. \\

We will denote points in $\textbf{E}$ by $[v,i]$, where $v \in E_i$. Strictly speaking this is an abuse of notation, since the equivalence classes of $\textbf{E}$ are different from the equivalence classes used to define points in $\textbf{M}$. However, in this notation the projection map $\boldsymbol{\pi}: \textbf{E} \rightarrow \textbf{M}$ can be easily defined as  $\boldsymbol{\pi}([v,i]) = [\pi_i(v), i]$. Observe that $\boldsymbol{\pi}$ is well defined since our requirement that the bundle morphisms $F_{ij}$ cover the gluing map $f_{ij}$ ensures that 
$$ \boldsymbol{\pi}([F_{ij}(v),j]) =  [\pi_j \circ F_{ij}(v), j] = [f_{ij} \circ \pi_i(v), j] = [\pi_i(v), i]    $$
 for all $v$ in $E_{ij}$. Moreover, the map $\boldsymbol{\pi}$ is manifestly smooth since its local expression around any point $[x,i]$ of $\textbf{M}$ will be the composition $\phi_i \circ \pi_i \circ \chi_i^{-1}$. \\
 
By construction the map $\boldsymbol{\pi}$ is surjective, and furthermore we may endow the preimages $\boldsymbol{\pi}^{-1}([x,i])$ with the structure of a rank-$k$ vector space induced from the fibre $\pi_i^{-1}(x)$ of $E_i$. In our notation, addition and scalar multiplication are given by 
$$ [v,i] + [w,i] = [v+w, i] \ \text{and} \   \lambda[v,i] = [\lambda v,i], $$
respectively. These operations are well-defined by our assumption that the $F_{ij}$ are bundle morphisms, and consequently the fibres of $\textbf{E}$ will indeed be $k$-dimensional vector spaces. \\

In direct analogy to the construction of smooth atlases in Section 1.2, we can describe local trivialisations of $\textbf{E}$ using the bundles $E_i$. Suppose that we have a point $[x,i]$ in $\textbf{M}$, and fix $U$ to be a local trivialisation of the bundle $E_i$ at the point $x$, with trivialising map $\Theta$. Since $\phi_i$ is an open map, we can consider the set $\phi_i(U)$ as an open neighbourhood of $[x,i]$ in $\textbf{M}$. This data can be arranged into the following diagram
\begin{center}
        % https://q.uiver.app/?q=WzAsNixbMiwxLCJVIl0sWzEsMSwiXFxwaGlfaV57LTF9KFUpIl0sWzEsMCwiXFxwaV9pXnstMX0gXFxjaXJjIFxccGhpX2leey0xfShVKSJdLFswLDAsIlxccGhpX2leey0xfShVKSBcXHRpbWVzIFxcbWF0aGJie1J9XmsiXSxbMiwwLCJcXGJvbGRzeW1ib2x7XFxwaX1eey0xfShVKSJdLFszLDAsIlUgXFx0aW1lcyBcXG1hdGhiYntSfV5rIl0sWzIsNCwiXFxjaGlfaSIsMCx7InN0eWxlIjp7InRhaWwiOnsibmFtZSI6Imhvb2siLCJzaWRlIjoidG9wIn19fV0sWzQsMCwiXFxib2xkc3ltYm9se1xccGl9IiwyLHsic3R5bGUiOnsiaGVhZCI6eyJuYW1lIjoiZXBpIn19fV0sWzEsMCwiXFxwaGlfaSIsMix7InN0eWxlIjp7InRhaWwiOnsibmFtZSI6Imhvb2siLCJzaWRlIjoidG9wIn19fV0sWzMsMSwicF8xIiwyLHsic3R5bGUiOnsiaGVhZCI6eyJuYW1lIjoiZXBpIn19fV0sWzIsMywiXFxQaGkiLDJdLFs0LDUsIlxcUHNpIiwwLHsic3R5bGUiOnsiYm9keSI6eyJuYW1lIjoiZGFzaGVkIn19fV0sWzUsMCwicF8xIiwwLHsic3R5bGUiOnsiaGVhZCI6eyJuYW1lIjoiZXBpIn19fV0sWzIsMSwiXFxwaV9pIiwwLHsic3R5bGUiOnsiaGVhZCI6eyJuYW1lIjoiZXBpIn19fV1d
\begin{tikzcd}[row sep=4.2em]
	{U \times \mathbb{R}^k} & {\pi_i^{-1}(U)} & {\boldsymbol{\pi}^{-1}(\phi_i(U))} & {\phi_i(U) \times \mathbb{R}^k} \\
	& {U} & {\phi_i(U)}
	\arrow["{\chi_i}", from=1-2, to=1-3]
	\arrow["{\boldsymbol{\pi}}"', , from=1-3, to=2-3]
	\arrow["{\phi_i}"',  from=2-2, to=2-3]
	\arrow["{p_1}"', from=1-1, to=2-2]
	\arrow["\Theta"', from=1-2, to=1-1]
	\arrow["\Psi", dashed, from=1-3, to=1-4]
	\arrow["{p_1}",  from=1-4, to=2-3]
	\arrow["{\pi_i}",  from=1-2, to=2-2]
\end{tikzcd}
    \end{center}
where the $p_1$ are projections onto the first factor. The trivialising map $\Psi$ can then be defined as the composition $ \Psi := (\phi_i,\text{id}) \circ \Theta\circ\chi_i^{-1}$. Transition functions for local trivialisations around points in the gluing regions $M_{ij}$ will be: 
\begin{align*}
        \Psi_{\beta} \circ \Psi_\alpha^{-1} & = \left((\phi_j, \text{id}) \circ \Theta_\beta \circ \chi_j^{-1} \right) \circ \left((\phi_i, \text{id}) \circ \Theta_\alpha \circ \chi_i^{-1} \right)^{-1} \\
        & = \left((\phi_j, \text{id}) \circ \Theta_\beta\right) \circ\left(  \chi_j^{-1}  \circ  \chi_i \right)\circ\left(  \Theta_\alpha^{-1} \circ (\phi_i^{-1}, \text{id})\right) \\
        & = \left((\phi_j, \text{id}) \circ \Theta_\beta \circ F_{ij} \circ \Theta_\alpha^{-1} \circ (\phi_i^{-1}, \text{id})\right),
    \end{align*}
which mimic the local properties of the bundle morphisms $F_{ij}$. \\

According to our discussion thus far, we may consider $\textbf{E}$ as a rank-$k$ vector bundle fibred over the non-Hausdorff manifold $\textbf{M}$ in which the maps $\chi_i$ of $\textbf{E}$ are injective bundle morphisms that cover the canonical embeddings $\phi_i$. This is summarised in the following result. 
\begin{theorem}\label{THM: colimit of bundles} 
Let $\mathcal{G}:= (\textsf{E}, \textsf{B}, \textsf{F})$ be a triple of sets in which: 
\begin{enumerate}
    \item $\textsf{E} = \{ (E_i, \pi_i, M_i) \}_{i \in I}$ is a collection of rank-$k$ vector bundles,
    \item $\textsf{B} = \{ E_{ij} \}_{i,j \in I}$ consists of the restrictions of the bundles $E_i$ to the intersections $M_{ij}$, and
    \item $\textsf{F} = \{F_{ij} \}_{i,j \in I}$ is a collection of bundle isomorphisms $F_{ij}: E_{ij} \rightarrow E_{ji}$ that cover the gluing maps $f_{ij}$ and satisfy the condition $F_{ik} = F_{jk}\circ F_{ij}$ on the intersections $M_{ijk}$, for all $i,j,k$ in $I$.
\end{enumerate}
Then the resulting adjunction space $\textbf{E} := \bigcup_{\mathcal{G}} E_i$ has the structure of a non-Hausdorff rank-$k$ vector bundle over $\textbf{M}$ in which the canonical inclusions $\chi_i: E_i \rightarrow \textbf{E}$ are bundle morphisms covering the canonical embeddings $\phi_i:M_i \rightarrow \textbf{M}$.
\end{theorem}

We will now confirm that the bundle $\textbf{E}$ described satisfies a certain universal property.
\begin{theorem}\label{THM: universal property of colimit of bundles}
Let $\textbf{E}$ be a vector bundle over $\textbf{M}$ as in Theorem \ref{THM: colimit of bundles}, and let $\textbf{F} \xrightarrow{\boldsymbol{\rho}}\textbf{M}$ be a vector bundle. If there exist bundle morphisms $\xi_i: E_i \rightarrow\textbf{F}$ covering the canonical maps $\phi_i$ satisfying $\xi_i = \xi_j \circ F_{ij}$ for all $i,j$ in $I$. Then there exists a unique bundle morphism $\alpha: \textbf{E} \rightarrow \textbf{F}$ such that $\xi_i = \chi_i \circ \alpha$ for all $i$ in $I$. 
\end{theorem}
\begin{proof}
Since all bundles are smooth manifolds and all bundle morphisms are smooth maps, we may apply the universal property of smooth non-Hausdorff manifolds to conclude that there exists a unique smooth map $\alpha$ from $\textbf{E}$ to $\textbf{F}$ defined by $$ \alpha([v,i]) = \xi_i(v).      $$ 
Observe that since the maps $\chi_i: E_i\rightarrow \textbf{E}$ and the maps $\xi_i: E_i \rightarrow \textbf{F}$ both cover the canonical embeddings $\phi_i$, we have that 
$$  \boldsymbol{\rho} \circ \alpha( [v,i]) =  \boldsymbol{\rho} \circ \xi_i(v) = \phi_i \circ \pi_i(v) = \boldsymbol{\pi}\circ\chi_i(v) = \boldsymbol{\pi}([v,i]),    $$
and thus the map $\alpha$ covers the identity map on $\textbf{M}$. Moreover, $\alpha$ acts linearly on fibres of $\textbf{E}$ since $\alpha$ coincides with the map $\xi_i \circ \chi_i^{-1}$ on each $E_i$, from which we may conclude that $\alpha$ is a bundle morphism.
\end{proof}

According to the above result, we may interpret any vector bundle $\textbf{E}$ constructed according to Theorem \ref{THM: colimit of bundles} as a colimit in the category of smooth vector bundles over locally-Euclidean, second-countable spaces. 

\subsection{A Reconstruction Theorem}
It is well-known that all non-Hausdorff manifolds can be constructed using adjunction spaces. In essence, this result follows from the fact that maximal Hausdorff submanifolds of a given non-Hausdorff manifold form an open cover \cite{hajicek1971causality}. It is then possible to fix a minimal open cover by Hausdorff submanifolds, and then to glue them along the identity maps defined on the pairwise intersections. The details of this result can be found in \cite{luc2020interpreting, oconnell2023non} and \cite{OConnellthesis} in different forms. \\

We will now argue that all vector bundles over $\textbf{M}$ are colimits in the sense of Theorem \ref{THM: colimit of bundles}. The argument is similar to the manifold case: we can always restrict a bundle down to the component spaces $M_i$ to create a collection of Hausdorff bundles that can then be re-identified. Formally, this restriction is obtained by taking the pullbacks of the bundle $\textbf{E}$ along the canonical embeddings $\phi_i$. 
%As such, before getting to the result we will first recall the definition of a pullback bundle. \\

%Given a smooth map $\textbf{f}: \textbf{M}\rightarrow \textbf{N}$ between manifolds (Hausdorff or otherwise) and a bundle $(\textbf{E},\boldsymbol{\pi},\textbf{N})$, we define the pullback of $\textbf{E}$ along $\textbf{f}$ as: 
%$$ \textbf{f}^*\textbf{E} := \{ (x,v) \in \textbf{M}\times \textbf{E} \ | \ \textbf{f}(x)=\boldsymbol{\pi}(v)   \}.$$
%The projection map $p_1: \textbf{f}^*\textbf{E} \rightarrow \textbf{M}$ is taken to be the projection onto the first factor, and the projection $p_2:\textbf{f}^*\textbf{E}\rightarrow \textbf{E}$ onto the second factor is a bundle morphism that covers the smooth map $\textbf{f}$. With this in mind, we can now prove our desired result. 

\begin{theorem}\label{THM: reconstruction theorem for colimit bundles}
Let $\textbf{E}$ be some vector bundle over $\textbf{M}$. Then $\textbf{E}$ is isomorphic to a colimit bundle of the form detailed in Theorem \ref{THM: colimit of bundles}. 
\end{theorem}
\begin{proof} 
We would like to define a colimit bundle by gluing the pullback bundles $\phi_i^*\textbf{E}$ along bundle morphisms that act by identity on each fiber. Formally, this can be achieved by the data $\mathcal{G}=(\textsf{E},\textsf{B},\textsf{G})$, where:
\begin{itemize}
    \item $\textsf{E}$ consists of the pullback bundles $E_i := \phi_i^*\textbf{E}$,
    \item $\textsf{B}$ consists of the restricted bundles $E_{ij} := (\phi_i^*\textbf{E})|_{M_{ij}}\cong (\phi_i \circ \iota_{ij})^*\textbf{E}$, and
    \item $\textsf{G}$ consists of the maps $F_{ij}:(\phi_i\circ \iota_{ij})^*\textbf{E} \rightarrow  \phi_j^*\textbf{E}$ where $(x,v)\mapsto (f_{ij}(x),v)$.
\end{itemize}
This data satisfies the criteria of Theorem \ref{THM: colimit of bundles}, thus we may conclude that $\textbf{F}:=\adjg E_i$ is a vector bundle over $\textbf{M}$. By construction, the pullback bundles $\phi_i^*\textbf{E}$ cover the canonical embeddings $\phi_i$ via the maps $p_2$ which project onto the second factor of the Cartesian product. This means that pairwise we have the following diagram. 
\begin{center}
\begin{tikzcd}[row sep=1.9em, column sep=0.8em]
 & (\phi_i^*\textbf{E})|_{M_{ij}} \arrow[ld, "F_{ij}"'] \arrow[dd, "p_1"', near start] \arrow[rr, "p_2"] &  & \phi_i^*\textbf{E} \arrow[dd, "p_1"] \arrow[ld, "p_2"'] \\
\phi_j^*\textbf{E} \arrow[rr, "p_2", near end, crossing over] \arrow[dd, "p_1"'] &  & \textbf{E}  &  \\
 & M_{ij} \arrow[ld, "f_{ij}"', near start] \arrow[rr, "\iota_{ij}", near start] &  & M_i \arrow[ld, "\phi_i"] \\
M_j \arrow[rr, "\phi_j"'] &  & \textbf{M} \arrow[from=uu, "\boldsymbol{\pi}", near start, crossing over] & 
\end{tikzcd} 
\end{center}
This diagram commutes since the morphisms $F_{ij}$ act by the identity on their second factors. According to the universal property of the colimit bundle $\textbf{F}$ we may induce a (unique) bundle morphism $\alpha$ from $\textbf{F}$ to $\textbf{E}$. Pointwise, the map $\alpha$ acts by $[(x,v),i] \mapsto v$. This map is clearly bijective, from which it follows that $\alpha$ is a bundle isomorphism.
\end{proof}

\subsection{Sections}

In Theorem \ref{THM: fibred product of smooth functions} we saw that the ring of smooth functions of $\textbf{M}$ is naturally isomorphic to the fibred product $\prod_{\mathcal{F}} C^\infty(M_i)$. In categorical terms, this product can be seen as the limit of a diagram that is formed by applying the $C^\infty$ functor to all of the data in $\mathcal{F}$. We will now extend this result to sections of arbitrary bundles over $\textbf{M}$. Throughout this section we take $\textbf{E}$ to be an arbitrary but fixed vector bundle over $\textbf{M}$, and we will denote by $E_i$ the restricted (Hausdorff) bundles over the subspaces $M_i$. In analogy to Lemma \ref{LEM: Adjunction of a smooth function}, we will first show that every section of the vector bundle $\textbf{E}$ can be described by gluing sections of $E_i$ that are compatible on overlaps.

\begin{lemma}\label{LEM: sections can be glued}
For each $i$ in $I$, let $s_i$ be a section of $E_i$. If the equality $F_{ij} \circ s_{i} = s_j \circ f_{ij}$ holds for all $i,j$ in $I$ then the function $\textbf{s}: \textbf{M}\rightarrow \textbf{E}$ defined by $\textbf{s}([x,i]) = [s_i(x), i]$ is a smooth section of $\textbf{E}$.
\end{lemma}
\begin{proof}
Observe first that $\textbf{s}$ is well-defined, since: 
$$  \textbf{s}([f_{ij}(x),j]) = [s_j \circ f_{ij}(x), j] = [F_{ij} \circ s_{i}(x), j] = [s_i(x),i].     $$ 
The map $\textbf{s}$ is a right-inverse of the projection map $\boldsymbol{\pi}$ since $\boldsymbol{\pi}\circ \textbf{s}([x,i]) = \boldsymbol{\pi}([s_i(x),i]) = [x,i]$ for all $[x,i]$ in $\textbf{M}$. Finally, since the restriction of $\textbf{s}$ to each $M_i$ equals $\chi_i \circ s_i \circ \phi_i^{-1}$, the map $\textbf{s}$ is smooth by Lemma \ref{LEM: smooth maps}.
\end{proof}

%In the Hausdorff setting, the space of sections $\Gamma(E)$ of a vector bundle $E\xrightarrow{\pi}M$ can be interpreted as a module over the algebra $C^\infty(M)$. Abstractly, we may interpret $\Gamma$ as a contravariant functor from the category of smooth vector bundles to the category of modules over $C^\infty(M)$. 

Following on from the approach of Section 1.3.1, we may now prove an analogue to Theorem \ref{THM: functions non-empty}.
\begin{theorem}\label{THM: non-zero sections exist}
Any section $s_i$ of $E_i$ can be extended to a section $\textbf{s}$ of $\textbf{E}$. 
\end{theorem}
\begin{proof}
We will proceed as in Theorem \ref{THM: functions non-empty}, that is, by induction on the size of indexing set $I$. Suppose first that $\textbf{M}$ is a binary adjunction space $M_1 \cup_{f_{12}} M_2$, with $E$ isomorphic to $E_1 \cup_{F_{12}} E_2$. As in Theorem \ref{THM: functions non-empty}, we may take a non-zero section on $E_2$, restrict it to the closure $Cl^{M_2}(M_{12})$, and then use the diffeomorphism $\overline{f_{12}}$ to pull back $s_2$ to some smooth section $s_{12}$ defined on the submanifold $Cl^{M_1}(M_{12})$ of $M_1$. Using the fact that $F_{12}$ is a bundle isomorphism onto its image, we may interpret $s_{12}$ as a section of the closed subbundle $Cl^{\textbf{E}}(E_{12})$ defined over $Cl^{M_1}(M_{12})$.

In order to extend $s_{12}$ into the rest of $M_1$, we will need to apply a generalisation of the Extension Lemma for sections of vector bundles  (cf. Lemma 10.12 of \cite{lee2013smooth}). The idea behind this generalised Extension Lemma is essentially the same as in the case of smooth functions -- we may always endow the closure $Cl^{M_1}(M_{12})$ with an outward-pointing collar neighbourhood $U$, and then use the flow of the associated vector field to extend $s_{12}$ to all of $U$. Using a partition of unity subordinate to the open cover $\{U, M_1 \backslash Cl^{M_1}(M_{12}) \}$ of $M_1$, we may then create a section $s_1$ of $E_1$.

Observe that by construction, any section $s_1$ created according to the above procedure will restrict to $s_{12}$ on $M_{12}$. A global section $\textbf{s}$ of $\textbf{E}$ then exists by applying of Lemma \ref{LEM: sections can be glued} to $s_1$ and $s_2$. The inductive case follows the same structure as Theorem \ref{THM: functions non-empty}, this time using the Extension Lemma for vector bundles instead.
\end{proof}
Using the above, we can now create an argument similar to that of Theorem \ref{THM: fibred product of smooth functions}. 
\begin{theorem}\label{THM: fibred product of sections}
For any vector bundle $\textbf{E}$ over $\textbf{M}$, we have that $$  \Gamma(\textbf{E}) \cong \prod_{\mathcal{F}} \Gamma(E_i).   $$ 
\end{theorem}
\begin{proof}
The argument is the same as that of Theorem \ref{THM: fibred product of smooth functions}, except that this time we use Lemma \ref{LEM: sections can be glued} and Theorem \ref{THM: non-zero sections exist}. 
\end{proof}

\subsection{Riemannian Metrics}
In order to discuss Čech cohomology in the next section, we will first need to confirm that metrics exist on arbitrary vector bundles fibred over $\textbf{M}$. The precise construction of such metrics will be similar to the approach of Lemmas \ref{THM: functions non-empty} and \ref{LEM: sections can be glued}. However, in order to use this technique we first need to confirm the following. 

\begin{lemma}\label{LEM: tensor bundle colimit}
Let $\textbf{E}$ be a vector bundle over $\textbf{M}$ with colimit representation $\bigcup_{\mathcal{G}} E_i$. Then the $(0,2)$ tensor bundle $T^{(0,2)}\textbf{E}$ is isomorphic to the colimit bundle $$ \bigcup_{\mathcal{G}} T^{(0,2)}E_i .     $$
\end{lemma}
\begin{proof}
For readability we will only provide a sketch, since the details of this argument can already be found in \cite{OConnellthesis}. Let us denote by $F_{ij}$ the bundle morphisms that are used to construct $\textbf{E}$ from the $E_i$. Since the $F_{ij}$ are diffeomorphisms from $E_{ij}$ to $E_{ji}$, the differentials $dF_{ij}$ will be bundle isomorphisms between the tangent bundles $TE_{ij}$ and $TE_{ji}$. Moreover, a basic property of differentials confirms that  
 $$   dF_{ik} = d(F_{jk} \circ F_{ij}) = dF_{jk} \circ d F_{ij}.$$
Consequently, we may glue each tangent bundle $TE_i$ along the differentials $dF_{ij}$ to create the bundle $\adjg TE_i$. The differentials $d\chi_i:TE_i \rightarrow T\textbf{E}$ of the canonical embedding maps $\chi_i: E_i \rightarrow \textbf{E}$ can then be used together with Theorem \ref{THM: universal property of colimit of bundles} to conclude that the bundle $\adjg TE_i$ is isomorphic to $T\textbf{E}$. A similar argument can be made for the bundle $T^{(0,2)}\textbf{E}$, except that this time we glue along the maps that pull back the $(0,2)$-tensors along the diffeomorphisms $F_{ij}$.    
\end{proof}
A metric tensor on $\textbf{E}$ may be seen as a global non-vanishing section of the bundle $T^{(0,2)}\textbf{E}$ that is symmetric and positive-definite in its local expression.  
Using the above result, we may readily construct metrics on any vector bundle fibred over $\textbf{M}$.

\begin{theorem}\label{THM: E admits a bundle metric}
    Any vector bundle $\textbf{E}$ over $\textbf{M}$ admits a metric. 
\end{theorem}
\begin{proof}
According to Theorem \ref{THM: reconstruction theorem for colimit bundles}, we may view $\textbf{E}$ as a colimit of bundles $E_i$ that are fibred over the Hausdorff submanifolds $M_i$. Lemma \ref{LEM: tensor bundle colimit} then allows us to express the tensor bundle $T^{(0,2)}\textbf{E}$ as a colimit of the tensor bundles $T^{(0,2)}E_i$. We may then apply the construction of Theorem \ref{THM: non-zero sections exist} and use an inductive series of partitions of unity defined on each $M_i$ in order to construct a global section $\textbf{g}$ of the bundle $T^{(0,2)}\textbf{E}$. Note that we may guarantee that $\textbf{g}$ is a bundle metric if we start with bundle metrics $g_i$ of $E_i$ and use the fact that bundles metrics are closed under convex combinations. 
\end{proof}

In the next section we will need to appeal to the existence of Riemannian metrics defined on $\textbf{M}$. Fortunately, these exist as an application of the previous result to $T\textbf{M}$.

\begin{corollary}\label{COR: M admits Riemannian metric}
$\textbf{M}$ admits a Riemannian metric. 
\end{corollary}

\section{Čech Cohomology and Line Bundles}

Theorems \ref{THM: colimit of bundles} and \ref{THM: universal property of colimit of bundles} tell us that any line bundle over $\textbf{M}$ exists as a colimit of line bundles defined on each of the submanifolds $M_i$. In this section we will explore how this relationship manifests in the language of Čech cohomology. To begin with, we will proceed generally and study the Čech cohomology of $\textbf{M}$ in terms of the cohomologies of the $M_i$.\\

Before getting to any results, we will first briefly recall the formalism of Čech cohomology. Aside from a slight change in notation, we will closely follow 
\cite{bott1982differential}. Consider an arbitrary topological space $X$, with an open cover $\mathcal{U} := \{ U_\alpha \ | \ \alpha \in A\}$, and let $G$ be an Abelian group. We will use index notation to abbreviate multiple intersections of open sets in $\mathcal{U}$, that is, we will write $ U_{\alpha_0 \cdots \alpha_q} := U_{\alpha_0} \cap \cdots U_{\alpha_q}$. A degree-$q$ Čech cochain $\check{f}$ consists of a choice of constant functions from the $(q+1)$-ary intersections of elements of $\mathcal{U}$ to the group $G$. In symbols: 
$$  \check{f} := \big{\{} f_{\alpha_0 \cdots \alpha_q}: U_{\alpha_0 \cdots \alpha_q} \rightarrow G \ \big{|} \ \alpha_0 < \cdots < \alpha_q \in A \ \textrm{and} \ f_{\alpha_0 \cdots \alpha_q} \ \text{is constant} \big{ \}}. 
$$
The space of Čech $q$-cochains, which we will denote by $\check{C}^q(X, \mathcal{U}, G)$, consists of all sets $\check{f}$ of the above form. This space naturally inherits an Abelian group structure from $G$. We denote by $\delta$ the Čech differential, which raises the degree of each cochain by one. In additive notation, the map $\delta$ acts as follows: 
$$  \delta f_{\alpha_0 \cdots \alpha_q \alpha_{q+1}}(x) =  \sum_{i} (-1)^i f_{\alpha_0 \cdots \hat{\alpha_i} \cdots \alpha_{q+1}}(x),    $$ where here the caret notation $\hat{\alpha_i}$ denotes exclusion of that index. 
The Čech differential is a group homomorphism that squares to zero, and we denote the resulting cohomology groups by $\check{H}^q(X, \mathcal{U}, G)$. We may define these groups for any open cover $\mathcal{U}$, and the collection of all $\check{H}^q(X, \mathcal{U}, G)$ can be made into a directed system of groups once ordered by refinement. The Čech cohomology of $X$ is then defined as the direct limit: $$    \check{H}^q(X) := \lim_{\mathcal{U}} \check{H}^q(X, \mathcal{U}).    $$
In what follows we will need to make use of the pullback of Čech cochains, so we recall this notion now. Suppose that $\varphi: X\rightarrow Y$ is a continuous map and $\mathcal{U}$ is an open cover of the topological space $Y$. We may define an open cover $\mathcal{V}$ of $X$ by considering all sets of the form $\varphi^{-1}(U)$, where $U$ is an element of $\mathcal{U}$. We may then define a map $\varphi^*: \check{C}^q(Y, \mathcal{U}, G)\rightarrow \check{C}^q(X, \mathcal{V}, G)$ by demanding that $\varphi^*f_{\alpha_0 \cdots \alpha_q}(x) = f_{\alpha_0 \cdots \alpha_q}\circ \varphi(x)$ for all locally-constant functions on $Y$. By construction, the pullback $\varphi^*$ is a group homomorphism.

\subsection{Čech Cohomology via a Mayer-Vietoris Sequence}
We will now set about expressing the Čech cohomology $\check{H}^q(\textbf{M}, G)$ in terms of the groups $\check{H}^q(M_i, G)$. We will obtain this relationship inductively, so throughout this section we will assume that $\textbf{M}$ can be expressed as the colimit of two Hausdorff manifolds $M_1$ and $M_2$, in accordance with Remark \ref{REM: assumptions in this paper}.\footnote{Although we are working with non-Hausdorff manifolds primarily, it should be noted that the following derivation will also work for more general colimits of Hausdorff manifolds that do not assume the ``homeomorphic boundary" condition.} We will also assume that $\textbf{M}$ is endowed with a fixed but arbitrary open cover $\mathcal{U}$. We will not need to appeal to the particular structure of the Abelian group $G$, so we suppress this in our notation. \\

We start with a derivation of a Mayer-Vietoris sequence for $\textbf{M}$. In order to do so, we will need to make use of the pullbacks of cochains. According to our configuration, we have the following commutative diagram of pullbacks 

\begin{center}
    % https://tikzcd.yichuanshen.de/#N4Igdg9gJgpgziAXAbVABwnAlgFyxMJZABgBpiBdUkANwEMAbAVxiRAFkB9YARgCYAviAGl0mXPkIo+5KrUYs2XHsNEgM2PASJk+c+s1aIOnPqrGbJRGXuoHFxgDqOcMAB44ARgDNg7ISIWEtooAKyydgpGIM4AxgAWMLEA1sAAwgIAeskAFFy8gqQABM4AtnQ48bGMwACqWQUCAJRF5uriWlLI4bbyhmxxiSnpWbn5hSWO5ZXVDHUNgi1tGsFdAOwRfQ4xjglJqRnZedw8IpPTVTX1mbzNrYHtliHIG7320YP7I0fOrh4+fjOzk8EAYUDgAE9SiC5mUKpc5vVmssOlYUABmUg8fRRNgop5dAAsWJx-WMwjkMCgAHN4ERQN4AE4QUpIMggHAQJAyLbRbwnQQgagMOieGAMAAKqJCIEZWGp8RwbSZLO51E5SExvIGjjQ8SwpiFIBFYsl0qksvliuVzNZiB46q5iC17x1eoNKgeKrt7I19siZJ2+BwdAFATU3qQADZHUhidqnI5g6HGpkAFRGk3iqUEthyhVKr22pDhDlO+Ou4z81MZouqxAbMvRgPbZzuzg8dOZ0XZ815q2FiPFhuxxClys7dt8Lt1u0ADlHAE51XQsAw2OU0HANS2PrshgdRhSBEA
\begin{tikzcd}[column sep=1.8em]
M_{12} \arrow[dd, "f_{12}"'] \arrow[rr, "\iota_{12}"] &  & M_1 \arrow[dd, "\phi_1"] &                                      &    & {\check{C}^q(M_{12}, \mathcal{U}^{12}) }                      &  & {\check{C}^q(M_{1}, \mathcal{U}^{1}) } \arrow[ll, "\iota_{12}^*"']                                 \\
                                                      &  &                          & {} \arrow[r, "\check{C}^q", maps to] & {} &                                                               &  &                                                                                                    \\
M_2 \arrow[rr, "\phi_2"']                             &  & \textbf{M}               &                                      &    & {\check{C}^q(M_{2}, \mathcal{U}^{2}) } \arrow[uu, "f_{12}^*"] &  & {\check{C}^q(\textbf{M}, \mathcal{U})} \arrow[uu, "\phi_1^*"'] \arrow[ll, "\phi_2^*"]
\end{tikzcd}
\end{center}

where here $\mathcal{U}^i = \{ \phi_i^{-1}(U) \ | \ U \in \mathcal{U}\}$, and $\mathcal{U}^{12}$ is defined similarly. As a convention we will index these three covers using the relevant subsets of the indexing set of $\mathcal{U}$. \\

We can combine the various pullback maps in order to create a single sequence from the above diagram. We will consider two maps: 
$ \Phi^*$, which acts on cochains on $\textbf{M}$ by concatenating the pullbacks $\phi_i^*$ (as in \ref{THM: fibred product of smooth functions}, \ref{THM: fibred product of sections}), and the map $\iota_{12}^* - f_{12}^*$, which pulls back a pair of cochains defined on the $M_i$ to the subset $M_{12}$ and then computes their difference. The following result confirms that this arrangement of functions forms a short exact sequence. 
\begin{lemma}\label{LEM: 
Binary MV exact sequence}
The sequence
$$ 0 \rightarrow \check{C}^q(\textbf{M}, \mathcal{U})\overset{\Phi^*}\longrightarrow \check{C}^q(M_1, \mathcal{U}^1) \oplus \check{C}^q(M_2, \mathcal{U}^2) \xrightarrow{\iota_{12}
^*-f_{12}^*} \check{C}^q(M_{12}, \mathcal{U}^{12})\rightarrow 0         $$
is exact for all $q$ in $\mathbb{N}$.
\end{lemma}
\begin{proof}
The map $\Phi^*$ is injective since any two cochains on $\textbf{M}$ must differ somewhere on $M_1$ or $M_2$, thus they will differ once pulled back by the appropriate $\phi_i$ map. The surjectivity of the difference map $\iota_{12}^* - f_{12}^*$ follows from a standard ``extension by zero" argument: any cochain $\check{f}$ on $M_{12}$ is defined according to open sets that are restrictions of larger open sets in $\mathcal{U}^1$ and $\mathcal{U}^2$. For $\iota^*_{12}$ we may define $$  \check{g}_{\alpha_0 \cdots \alpha_q}(x) = \begin{cases}
    \check{f}_{\alpha_0 \cdots \alpha_q}(x) & \textrm{if} \ U_{\alpha_0 \cdots \alpha_q} \cap M_{12} = W_{\alpha_0 \cdots \alpha_q} \\
    0 & \text{otherwise}.
\end{cases}         $$ The cochain $\check{g}$ will then restrict to $\check{f}$ under the pullback by $\iota_{12}$.\\

Exactness is thus proved with an argument that $Im(\Phi^*)=\ker{(\iota_{12}^* - f_{12}^*)}$. Observe first that the commutativity of the diagram preceding this Lemma ensures that $\iota_{12}^* \circ \phi_1^* = f_{12}^*\circ \phi_2^*$, and thus $Im(\Phi^*)\subseteq \ker{(\iota_{12}^* - f_{12}^*)}$. For the converse inclusion, suppose that $\check{f}$ and $\check{g}$ are cochains on $M_1$ and $M_2$, respectively, such that $\iota_{12}^*\check{f} = f_{12}^*\check{g}$. Spelling this out, this means that 
\begin{equation}\tag{$*$}
    f_{\alpha_0 \cdots \alpha_q} \circ \iota_{12}(x) = g_{\alpha_0 \cdots \alpha_q} \circ f_{12}(x)
\end{equation}
for all open sets $W_{\alpha_0}, \cdots , W_{\alpha_q}$ in $\mathcal{U}^{12}$. In order to define a cochain $\check{h}$ on $\textbf{M}$ that satisfies $\Phi^*(\check{h}) = ( \check{f}, \check{g})$, we need to specify a locally-constant function for all of the $(q+1)$-ary intersections of open sets in $\mathcal{U}$. So, suppose that $U_{\alpha_0}, \cdots, U_{\alpha_q}$ are arbitrary elements of $\mathcal{U}$. For any $[x,i]$ in the intersection $U_{\alpha_0 \cdots \alpha_q}$, we define the map $h_{\alpha_0 \cdots \alpha_q}$ as follows: 
$$  h_{\alpha_0 \cdots \alpha_q}([x,i]) = \begin{cases}
    f_{\alpha_0 \cdots \alpha_q}\circ \phi_1^{-1}([x,i]) & \text{if} \ U_{\alpha_0 \cdots \alpha_q} \cap M_1 \neq \emptyset \\
    g_{\alpha_0 \cdots \alpha_q}\circ \phi_2^{-1}([x,i]) & \text{if} \ U_{\alpha_0 \cdots \alpha_q} \cap M_2 \neq \emptyset
\end{cases}.             $$
There is a potential ambiguity in this definition, so we must confirm that the value of $h_{\alpha_0 \cdots \alpha_q}$ does not depend on the choice of $M_1$ or $M_2$. So, suppose that the set $U_{\alpha_0 \cdots \alpha_q}$ intersects both $M_1$ and $M_2$. Then for any element in $M_{12}$, we have two representatives of the equivalent class $[x,1]$, namely $x$ and $f_{12}(x)$. We then have that 
$$  f_{\alpha_0 \cdots \alpha_q}\circ \phi_1^{-1}([x,1]) = f_{\alpha_0 \cdots \alpha_q}\circ \iota_{12}(x) \stackrel{(*)}{=} g_{\alpha_0 \cdots \alpha_q} \circ f_{12}(x) =  g_{\alpha_0 \cdots \alpha_q}\circ \phi_2^{-1}([f_{12}(x),2])   $$
as required. By construction, $\check{h}$ pulls back to $(\check{f}, \check{g})$ under the map $\Phi^*$. We may therefore conclude that $\ker{(\iota_{12}^* - f_{12}^*)} \subseteq Im(\Phi^*)$, from which the equality follows. 
\end{proof}

As a consequence of the above result, we observe that the Čech cochains of $\textbf{M}$ admit a fibred product structure.
\begin{corollary}\label{THM: Čech cochains are fibred product}
The Čech cochains on $\textbf{M}$ satisfy the equality 
\begin{align*}
\check{C}^q(\textbf{M}, \mathcal{U}) & \cong \check{C}^q(M_1, \mathcal{U}^1) \times_{\check{C}^q(M_{12}, \mathcal{U}^{12})} \check{C}^q(M_{2}, \mathcal{U}^2) \\
    & = \Big\{ \left(\check{f}^1, \check{f}^2\right) \in \check{C}^q(M_1, \mathcal{U}^1) \oplus \check{C}^q(M_2, \mathcal{U}^2) \ \Big| \ \iota_{12}^*\check{f}^1 = f_{12}^* \check{f}^2 \Big\}.
\end{align*}
for all $q\in \mathbb{N}$.
\end{corollary}

Routine computations verify that the maps $\Phi^*$ and $\iota_{12}^* - f_{12}^*$ commute with the Čech differential $\delta$. Consequently, we may expand the data of Lemma \ref{LEM: Binary MV exact sequence} into the long exact sequence
$$     \cdots \xrightarrow{\delta^*} \check{H}^q(\textbf{M}) \xrightarrow{\Phi^*} \check{H}^q(M_1) \oplus \check{H}^q(M_2) \xrightarrow{\iota_{12}^* - f_{12}^*} \check{H}^q(M_{12}) \xrightarrow{\delta^*} \check{H}^{q+1}(\textbf{M}) \rightarrow \cdots                 $$
where here the maps $\delta^*$ are the connecting homomorphisms induced as an application of the Snake Lemma. It should be noted that here we have tacitly passed to the cover-independent version of Čech cohomology. In fact, the relationship above is induced from a combination of refinement maps and the universal property of direct limits -- the details can be found in the Appendix.  \\

According to the above sequence, it will not be the case that the Čech cohomology of $\textbf{M}$ will equal a fibred product of the groups $\check{H}^q(M_1)$ and $\check{H}^q(M_2)$. Indeed, if we try and proceed in similar spirit to Theorems \ref{THM: fibred product of smooth functions} and \ref{THM: fibred product of sections} and argue for the injectivity of the $\Phi^*$, we will obtain an obstruction coming from the previous cohomology group of $M_{12}$. This obstruction can be read off of the Mayer-Vietoris sequence as: $$ \frac{\check{H}^q(\textbf{M})}{Im(\delta^*)} = \frac{\check{H}^q(\textbf{M})}{Ker(\Phi)} = Im(\Phi) = Ker(\iota_{12}^* - f_{12}^*) =  \check{H}^q(M_1) \times_{\check{H}^q(M_{12})} \check{H}^q(M_2).  $$
With an eye towards Section 3.4, we make the following observation. 
\begin{lemma}\label{LEM: Čech cohomology binary fibre product}
Suppose that $\textbf{M}$ is a binary adjunction space in which $M_{1}$, $M_2$ and $M_{12}$ are all connected. Then 
\begin{align*}
\check{H}^1(\textbf{M}) & \cong \check{H}^1(M_1)\times_{\check{H}^1(M_{12})}  \check{H}^1(M_2) \\
    & = \Big\{ \left([\check{f}^1], [\check{f}^2]\right) \in \check{H}^q(M_1) \oplus \check{H}^q(M_2) \ \Big| \ \iota_{12}^*[\check{f}^1] = f_{12}^* [\check{f}^2] \Big\}.
\end{align*}

\end{lemma}
\begin{proof}
According to Theorem 1.6 of \cite{oconnell2023non}, $\textbf{M}$ is connected. Since the $0^{\text{th}}$ Čech cohomology group of a connected topological space is always rank $1$, the Mayer-Vietoris sequence reduces to the following.
$$ 0 \rightarrow G \rightarrow G \oplus G \rightarrow G \xrightarrow{\delta^*} \check{H}^1(\textbf{M}) \rightarrow \check{H}^1(M_1)\oplus \check{H}^1(M_2) \rightarrow \check{H}^1(M_{12}) \rightarrow \cdots $$ 
Working from the left, exactness of this sequence implies that the labelled connecting homomorphism $\delta^*$ is the zero map, from which the result follows.
\end{proof}

\subsection{The General Case}

We will now discuss a version of Theorem \ref{LEM: Binary MV exact sequence} that is suitable for a non-Hausdorff manifold $\textbf{M}$ expressed as the colimit of finitely-many manifolds $M_i$. We will start with a generalisation of Lemma \ref{LEM: Binary MV exact sequence} that takes into account the multiple intersections $M_{i_1\cdots i_p}$. \\

To begin with, we need to generalise the difference maps $\iota_{12}^* - f_{12}^*$. Since there are now $n$-many submanifolds $M_i$, we will have multiple pairwise intersections $M_{ij}$. We may order the pairs of indices lexicographically and define a map 
$$  \tilde{\delta}: \bigoplus_i \check{C}^q(M_i, \mathcal{U}^i) \rightarrow \bigoplus_{i<j}\check{C}^q(M_{ij}, \mathcal{U}^{ij}) $$
which combines the difference maps $\iota_{ij}^* - f_{ij}^*$ in the obvious manner. 
Since there are now multiple intersections $M_{i_1 \cdots i_p}$, each with their own spaces of cochains $\check{C}^q(M_{i_1 \cdots i_q}, \mathcal{U}^{i_1 \cdots i_q})$, we will also have multiple pullbacks 
$$\iota_{i_1 \cdots \hat{i} \cdots i_p}^*: \check{C}^q(M_{i_1 \cdots \hat{i} \cdots i_p}, \mathcal{U}^{i_1 \cdots \hat{i} \cdots i_p}) \rightarrow \check{C}^q(M_{i_1 \cdots i_p}, \mathcal{U}^{i_1 \cdots i_p}),$$
where here $\iota_{i_1 \cdots \hat{i} \cdots i_p}: M_{i_1 \cdots i_p} \rightarrow M_{i_1 \cdots \hat{i} \cdots i_p}$ is the inclusion map that forgets the $i^{th}$ index. We can similarly define difference maps 
$$\tilde{\delta}: \bigoplus_{i_1 < \cdots < i_p} \check{C}^q(M_{i_1 \cdots i_p}, \mathcal{U}^{i_1 \cdots i_p}) \rightarrow 
\bigoplus_{i_1 < \cdots < i_{p+1}} \check{C}^q(M_{i_1 \cdots i_{p+1}}, \mathcal{U}^{i_1 \cdots i_{p+1}})$$
to act on cochains $\check{f}_{i_1 \cdots i_{p+1}}$ defined on each $M_{i_1 \cdots i_{p+1}}$ by 
$$  (\tilde{\delta}\check{f})_{i_1 \cdots i_{p+1}} = \sum_{i} (-1)^{i+1} \iota^*_{i_1 \cdots \hat{i} \cdots i_{p+1}} \check{f}_{i_1 \cdots \hat{i} \cdots i_{p+1}}.\footnote{Note that the extra sum in the exponent of $-1$ is due to our convention that the indexing set $I$ starts at one instead of zero.}$$ 

In essence, the map $\tilde{\delta}$ plays the role of the differential that one would define for the Čech complex constructed from the open cover $\{M_i\}$. In a similar manner to the binary case of Section 3.2, this map $\tilde{\delta}$ can be shown to square to zero and to commute with the Čech differential $\delta$ defined for $\check{C}^q(\textbf{M}, \mathcal{U})$. \\

We will now set about proving the general version of Lemma \ref{LEM: Binary MV exact sequence}, this time taking into account the multiple intersections $M_{i_1 \cdots i_p}$. 

\begin{theorem}\label{THM: generalised MV exact sequence}
The sequence $$ 0\rightarrow \check{C}^q(\textbf{M}, \mathcal{U}) \xrightarrow{\  \Phi^* \ } \bigoplus_i \check{C}^q(M_i, \mathcal{U}^i) \xrightarrow{\ \tilde{\delta}\ } \cdots \xrightarrow{ \ \tilde{\delta} \ } \check{C}^q(M_{1 \cdots n}, \mathcal{U}^{1 \cdots n}) \rightarrow 0     $$
is exact for all $q$ in $\mathbb{N}$.
\end{theorem}
\begin{proof}
We will proceed by induction on the size of the set $I$. The case of $I=2$ is already proved as Lemma \ref{LEM: Binary MV exact sequence}. To spare notation, we will illustrate the inductive argument for $I=3$, though it should be understood that the full inductive case is near-identical. With all the data available to us, we may construct a contravariant analogue of the diagram in \cite[Pg. 187]{bott1982differential} as follows, where we have suppress the open covers for readability. 
\begin{center}
\adjustbox{scale=0.8}{
% https://tikzcd.yichuanshen.de/#N4Igdg9gJgpgziAXAbVABwnAlgFyxMJZARgBpiBdUkANwEMAbAVxiRAB12BrGAJ2AAUnAMYALLAD0AVAEoAviDml0mXPkIoATOSq1GLNiNExhXYAGE50gQFkA+gGYZi5SAzY8BIg53V6zVkQOdjETM0tre2BiBzkZTgg0ZjgAAiMwiyspWztgTVjnJRUPdSIAFl89AMMQ41NMyNzifLiXYrUvFABWSv8DIIAGNrdVTw1kAd79QJAhopGSzonSTV0+mbnXdw7xntW-abZN9rGiHoc1w8Hh7dOUSYuD6uv529KUMn2q-uDQ+ojspwcDAAB44ABGADNgDZWq9Ru9kNovusauCsABzRLJXJYORpWoZAE5LCFLYIpY+FFXYLorFJJhwXEAHgAVvj0v8sjlgFh2WSToiKtTnr86uFuVFmgUbhTxmRHt8ZpyJY1ohyQkw0Ckopo4eTFuNtIrUUEVQ1slK4gkGalzcTdfrBZTSCaafbJU09QKFjtyq7LqLjr67iRSGVAz9g28ltoI08o7LDd5w5GNkm-SgKvGlUcM6GyAM03n4cmtKQiwn06XM8gfJXcy8DbWKg3TbNFLoYFAMfAiKBIbwIABbJA9EA4CBIAbzQcj6fUSdIMiN4JocSOaTDOejxArpeIbSroTsfA4OhNWLSUgpAC0kNyLWkPp3SCPB58q84Z4v0U4ojoHA8jkK8pDvAkf1yf9APVJ8pG3Idd0-A8ylnRCkAANkXKc9xnVxX1w7DlxXdtOAABXELc0PnPd9xw5oqxqPAGFgYBOFgBhzwUajdwYid6M-Uj2GY1j2JgTi6G4-D0L3ZCcIAdh4pAAE4iL3VDpJo4gKn45dx3bE8xCwOxiGvAkjLsTRnwQrTxwPYgsK-U8IHPL0snAh9oj1KjNN4xz7MU3zl3ktSAA4lMPd8cLwgcZMmXTcIi4h4vsnShKMrdqAYOhwXEsi5TYXhMVEHAbN4nT7PCoLDzk5ckrU4hiCSui9KSuz6OUiLNAqnC9WqqLlz62KtP83qYpAAi+PsjCuvapBYmq2rZKSpbiEC4beJChLNCaig5CAA
\begin{tikzcd}[row sep=4.8em, column sep = 1.4em]
            & 0 \arrow[d]                                                                 & 0 \arrow[d]                                                                                & 0 \arrow[d]                                                                                                    & 0 \arrow[d]                              &   \\
0 \arrow[r] & \ker{(\kappa^*)} \arrow[r, "\phi_3^*"] \arrow[d]                              & \check{C}^q(M_3) \arrow[r, "{(-f_{13}^*, -f_{23}^*)}"] \arrow[d]                        & \check{C}^q(M_{13})\oplus \check{C}^q(M_{23}) \arrow[r, "\tilde{\delta}"] \arrow[d] & \check{C}^q(M_{123}) \arrow[r] \arrow[d] & 0 \\
0 \arrow[r] & \check{C}^q(\textbf{M}, \mathcal{U}) \arrow[r, "\Phi^*"] \arrow[d, "\kappa^*"']            & \bigoplus_{i} \check{C}^q(M_i) \arrow[r, "\tilde{\delta}"] \arrow[d]                       & \bigoplus_{i<j} \check{C}^q(M_{ij}) \arrow[r, "\tilde{\delta}"] \arrow[d]                                      & \check{C}^q(M_{123}) \arrow[r] \arrow[d] & 0 \\
0 \arrow[r] & \check{C}^q(M_{1} \cup M_{2}) \arrow[r] \arrow[d] & \check{C}^q(M_{1})\oplus \check{C}^q(M_{2}) \arrow[r] \arrow[d] & \check{C}^q(M_{12}) \arrow[r] \arrow[d]                                                                        & 0 \arrow[r] \arrow[d]                    & 0 \\
            & 0                                                                           & 0                                                                                          & 0                                                                                                              & 0                                        &  
\end{tikzcd}

}
\end{center}
The map $\kappa^*$ is the pullback of the inclusion $\kappa: M_1 \cup_{f_{12}} M_2 \rightarrow \textbf{M}$ described in Section 1.1. The first row of the diagram is formed by taking kernels of the vertical maps, and the horizontal maps in the first row are defined so as to make the diagram commute. Since the columns of this diagram are all short exact sequences that split, in order to show the exactness of the center row it suffices to show that the first and third rows are exact. Observe first that the third row is a binary Mayer-Vietoris sequence, so is exact by Lemma \ref{LEM: Binary MV exact sequence}. 

Following the approach of \cite{crainic1999remark}, we may decompose the first row into sequences:

$$  0 \rightarrow \ker{(\kappa^*)} \overset{\phi_3^*}{\longrightarrow} \check{C}^q(M_3) \overset{-\iota^*}{\longrightarrow} \check{C}^q(M_{13} \cup M_{23}) \rightarrow 0        $$
and  
$$   0  \rightarrow \check{C}^q(M_{13} \cup M_{23}) \rightarrow \check{C}^q(M_{13})\oplus \check{C}^q(M_{23}) \rightarrow \check{C}^q(M_{123}) \rightarrow 0         $$
where here this decomposition is formed by taking the kernel of the $\tilde{\delta}$ map. The latter sequence is a Mayer-Vietoris sequence, so is exact by Lemma \ref{LEM: Binary MV exact sequence}. Thus our proof is complete once we argue that the former is a short exact sequence. 

Observe first that the map $\phi^*_3$ is injective since any two distinct elements of $\ker(\kappa^*)$ will be zero on $M_1 \cup M_2$, thus can only differ somewhere on $M_3$. The map $-\iota^*$ is surjective by a standard extension by zero argument (cf. Lemma \ref{LEM: Binary MV exact sequence}). Finally, we will show that $Im(\phi_3^*) = \ker(-\iota^*)$. The inclusion $Im(\phi_3^*) \subseteq \ker(-\iota^*)$ is guaranteed since any element in $\ker{(\kappa^*)}$ will also vanish on $M_{13} \cup M_{23}$. For the converse inclusion, let $\check{f}$ be some cochain on $M_3$ that restricts to zero on $M_{13}\cup M_{23}$. We can extend $\check{f}$ to a full cochain $\check{g}$ on $\textbf{M}$ by defining the functions of $\check{g}$ to be 
$$g_{\alpha_0 \cdots \alpha_q}([x,i]) = \begin{cases} f_{\alpha_0 \cdots \alpha_q}(x) & \text{if} \ x\in M_3 \\ 0 & \text{otherwise}\end{cases}.$$
The cochain $\check{g}$ will then satisfy $\phi_3^*\check{g} = \check{f}$, from which we may conclude that $\ker(-\iota^*)\subseteq Im(\phi_3^*) $, whence equality.
\end{proof}

The generalised Mayer-Vietoris long exact sequence of Theorem \ref{THM: generalised MV exact sequence} can be equivalently rephrased as follows.  

\begin{corollary}\label{COR: general Čech cochain fibre product}
The Čech cochains of $\textbf{M}$ satisfy the equality \begin{align*}
    \check{C}^q(\textbf{M}, \mathcal{U}) & = \prod_{\mathcal{F}} \check{C}^q(M_i, \mathcal{U}^i) \\
    & = \Big\{ \left(\check{f}^1, \cdots, \check{f}^n\right) \in \bigoplus_{i}\check{C}^q(M_i, \mathcal{U}^i) \ \Big| \ \iota_{ij}^*\check{f}^i = f_{ij}^* \check{f}^j \ \text{for all } i,j \in I  \Big\}.
\end{align*} 
\end{corollary}

In direct analogy to the construction of the Čech-de Rham bicomplex of \cite{bott1982differential}, we may arrange the cochain data of all the $M_i$ into a bicomplex as below. \vspace{-10mm}
\begin{center}
% https://tikzcd.yichuanshen.de/#N4Igdg9gJgpgziAXAbVABwnAlgFyxMJZARgBoBmAXVJADcBDAGwFcYkQAdDgIywHMIaFnAD6WAARcAxgAsYUgNbAAwgF8AegAYAFAFkxAShCrS6TLnyEUZAEzU6TVuy68BQ5qInS5ilRuJ6hsamIBjYeAREZMT2DCxsiJw8-ILCYpIcsvJKauo2gVhGJmbhlkQ2FLGOCUmuqR4iwFgAPABWqhlZvrk6+k3tRSFhFpEoFXY0cU6JLinuok1tHd7ZfuoBfVgDwSUjVsgVMZPVzsluaYvtnT45Gvmb28Wh5hH75JXH8ad1840trc0FMtMjc1r0-q0gYNdq8iO8Jg4vjMzvUFv9AcCurd1oF+lCds9SqNkO8jojprU5hd0UDrqtcvcIfinsNYSgACwfck1aRQCA4BAsl5lDmkBFTHmZPkCgmskXITlkiWnKTSwVDYXEsiaKpIpK0NWyzX7Co6z4UrgG-nqmHy95m7mnK0yoVE-aaUjs3UUo1u8qe701X17Ige8UnRLGewwKB8eBEUAAMwAThAALZID0gHAQJBkR3I2CMHD0Akp9N5mg5pAVAtJIslsupjOITnZ3OIACs5pqAFouA3S09yy33u2kG3lYl+xxB02K4gAGxVjsAdh7pznw+bSHX48QAA4N4WYMWhyER5X95OI0k8IxYMAB6eS6p5y3a9Wu8e71gHzAn1nF96DfbcF27fcjzrLh70fZ8z1Ai8d0QLMvzHKdf3-QDB0QpNkLHL9l2gjhYIA+DX3fCcV13H8YL-OCgIQyjED3L9iE0MCWygtjiE4pAiK-ABOPjEE-DtiBsESILY8gRO48T2RE4hpPEoiMIARxAGhGHobhTwABWNdhk34GQcGY5TqMQYg9wwtAo1UIA
\begin{tikzcd}[row sep=3.7em, column sep = 1.8em]
                                     \vdots                                                                       & \vdots                                                                               & \vdots                                                               &        \\
                                     \bigoplus_i \check{C}^2(M_i, \mathcal{U}^i) \arrow[r, "\tilde{\delta}"] \arrow[u]           & \bigoplus_{i<j} \check{C}^2(M_{ij}, \mathcal{U}^{ij}) \arrow[r, "\tilde{\delta}"] \arrow[u]            & \bigoplus_{i<j<k} \check{C}^2(M_{ijk}, \mathcal{U}^{ijk}) \arrow[r] \arrow[u]           & \cdots \\
{}                                   \bigoplus_i \check{C}^1(M_i, \mathcal{U}^i) \arrow[u, "\delta"] \arrow[r, "\tilde{\delta}"] & \bigoplus_{i<j} \check{C}^1(M_{ij}, \mathcal{U}^{ij}) \arrow[u, "-\delta"] \arrow[r, "\tilde{\delta}"] & \bigoplus_{i<j<k} \check{C}^1(M_{ijk}, \mathcal{U}^{ijk}) \arrow[u, "\delta"] \arrow[r] & \cdots \\
                                     \bigoplus_i \check{C}^0(M_i, \mathcal{U}^i) \arrow[u, "\delta"] \arrow[r, "\tilde{\delta}"] & \bigoplus_{i<j} \check{C}^0(M_{ij}, \mathcal{U}^{ij}) \arrow[u, "-\delta"] \arrow[r, "\tilde{\delta}"] & \bigoplus_{i<j<k} \check{C}^0(M_{ijk}, \mathcal{U}^{ijk}) \arrow[u, "\delta"] \arrow[r] & \cdots 
                                     
\end{tikzcd}
\end{center}

Theorem \ref{THM: generalised MV exact sequence} ensures that the rows of this bicomplex are exact. As such, we may employ standard spectral arguments (found in, say \cite{bott1982differential}) to conclude that the cohomology of this bicomplex coincides with the Čech cohomology of $\textbf{M}$. We remark that it is possible to obtain some alternate descriptions of the Čech cohomology of $\textbf{M}$ by computing the spectral sequence of the above bicomplex starting with the $\delta$-cohomology of the columns. However, given the scope of this paper we will not expand on this observation. For our purposes, we will only use the following result. 
\begin{theorem}\label{THM: 1st Čech cohomology is fibred product for connected colimits}
    Suppose that $\textbf{M}$ is built from $n$-many $M_i$ in which all intersections $M_{i_1 \cdots i_p}$ are connected for all $p\leq n$. Then $\check{H}^1(\textbf{M})$ coincides with the fibred product $\prod_{\mathcal{F}} \check{H}^1(M_i)$.
\end{theorem}
\begin{proof}
We will proceed by induction on the size of the indexing set $I$. This argument revolves around an inductive form of the finite fibred product, so the reader unfamiliar with this form is invited to read the appendix first. 

The binary case is already proved in the form of Lemma \ref{LEM: Binary MV exact sequence}. We will illustrate the case for $I=3$, though it should be understood that the inductive argument is near-identical to the following. Let $\textbf{M}$ be the colimit of the diagram formed from $M_1$, $M_2$ and $M_3$. We can use the inductive construction of $\textbf{M}$ as in Lemma \ref{LEM: adj of 3 manifolds is two-step} to view $\textbf{M}$ as an adjunction of the pair $M_1\cup M_2$ to $M_3$ along the union $M_{13} \cup M_{23}$. By the Mayer-Vietoris argument of Lemma \ref{LEM: Binary MV exact sequence} we have the following portion of the long exact sequence: 
\begin{center}
\adjustbox{scale=0.9}{\begin{tikzcd}[arrow style=math font,cells={nodes={text height=2ex,text depth=0.75ex}}]
       0 \arrow[r] & \check{H}^0 (\textbf{M}) \arrow[r] & \check{H}^0(M_{1} \cup M_2) \oplus  \check{H}^0(M_3) \arrow[r] \arrow[draw=none]{d}[name=Y, shape=coordinate]{} & \check{H}^0(M_{13} \cup M_{23}) \arrow[curarrow=Y]{dll}{} & \\
       & \check{H}^1 (\textbf{M}) \arrow[r] & \check{H}^1(M_{1} \cup M_2) \oplus  \check{H}^1(M_3) \arrow[r, "\iota^*_{a} - \iota^*_{b}"]  & \check{H}^1(M_{13} \cup M_{23}) \arrow[r] & \cdots 
\end{tikzcd}}
\end{center}
where we have used that $(M_1 \cup M_2)\cap M_3 = M_{13} \cup M_{23}$, and $\iota_a$ and $\iota_b$ are the inclusions of $M_{13} \cup M_{23}$ into $M_1 \cup M_2$ and $M_3$, respectively. Observe that again we have passed into the cover-independent Čech cohomology by using the direct limiting process found in the Appendix. By assumption all of the $\check{H}^0$ terms in this sequence equal $G$, so we may apply the same reasoning as that of Lemma \ref{LEM: Čech cohomology binary fibre product} to conclude that the connecting homomorphism $\delta^*$ is the zero map. Thus $\check{H}^1(\textbf{M}) = \ker(\iota^*_{a} - \iota^*_{b})$, which can be equivalently stated as: 
\begin{align*}
    \check{H}^1(\textbf{M}) & = \big\{ ([\check{f}], [\check{g}]) \in \check{H}^1(M_{1} \cup M_2) \oplus  \check{H}^1(M_3) \ \big{|} \  \iota^*_{a}[\check{f}] = \iota^*_{b}[\check{g}] \big\} \\
    & = \check{H}^1(M_{1} \cup M_2) \times_{\check{H}^1(M_{13} \cup M_{23})}  \check{H}^1(M_3).
\end{align*}
By assumption the triple intersection $M_{123}$ is also connected, so we apply the same reasoning to the union $M_{13} \cup M_{23}$ to conclude that 
$$ \check{H}^1(M_{13} \cup M_{23}) = \check{H}^1(M_{13})\times_{H^1(M_{123})} \check{H}^1(M_{23}).     $$
Putting this all together, we have that
\begin{align*}
    \check{H}^1(\textbf{M}) & = \check{H}^1(M_{1} \cup M_2) \times_{\check{H}^1(M_{13} \cup M_{23})}  \check{H}^1(M_3) \\
    & = \check{H}^1(M_{1} \cup M_2) \times_{\check{H}^1(M_{13})\times_{H^1(M_{123})} \check{H}^1(M_{23})}  \check{H}^1(M_3) \\ 
    & \cong \prod_{\mathcal{F}} \check{H}^1(M_i)
\end{align*}
where the final isomorphism follows from Theorem \ref{THM: Fibred product is Inductive Product}.
\end{proof}

\subsection{Classifying Line Bundles}
We will now discuss the prospect of classifying real line bundles over a fixed non-Hausdorff manifold $\textbf{M}$. Continuing on from the previous section, we will assume that $\textbf{M}$ may be expressed as the colimit of $n$-many Hausdorff manifolds $M_i$ according to the requirements of Remark \ref{REM: assumptions in this paper}. \\
We can describe any rank-$k$ vector bundle $\textbf{E}$ over $\textbf{M}$ as a choice of transition functions $g_{\alpha\beta}:U_{\alpha\beta} \rightarrow GL(k)$ satisfying the conditions $g_{\alpha\alpha} = id$ and $g_{\alpha \beta}  g_{\beta \gamma}  g_{\gamma \alpha} = id$. When this is the case, we say that the open cover $\mathcal{U} = \{ U_\alpha \}$ \textit{trivialises} the bundle $\textbf{E}$.  
As in the Hausdorff case, any two sets $\{g_{\alpha\beta} \}$ and $\{h_{\alpha \beta} \}$ of trivialisations will describe the same bundle $\textbf{E}$ whenever there exists a collection of maps $\{k_\alpha: U_\alpha \rightarrow GL(k) \}$ satisfying $h_{\alpha\beta} = k_\beta^{-1} g_{\alpha\beta} k_\alpha$. \\

%According to Theorem \ref{THM: reconstruction theorem for colimit bundles}, any vector bundle $\textbf{E}$ can be seen as a colimit of bundles fibred over the $M_i$. This can be equivalently encoded by the transition functions $g_{\alpha\beta}$ defined above. There will be two types of transition functions depending on whether or not the two open sets $U_{\alpha}$ and $U_{\beta}$ are generated from the same $M_i$. Thus, a choice of transition functions for all open sets in $\mathcal{U}$ will contain information on how to construct each Hausdorff bundle $E_i$, and how to glue these bundles together along each $M_{ij}$. \\

In the case that $\textbf{E}$ is a line bundle, we may use Theorem \ref{THM: E admits a bundle metric} to conclude that $\textbf{E}$ admits a bundle metric. This allows us to reduce the structure group of $\textbf{E}$ from $GL(1)$ down to $O(1)$. As such, any transition function will now take image in the Abelian group $\mathbb{Z}_2$. We may then equivalently view the construction of $\textbf{E}$ from local trivialisations as a cocycle in $\check{C}^1(\textbf{M}, \mathcal{U}, \mathbb{Z}_2)$ whose cohomology class determines $\textbf{E}$ up to isomorphism. Explicitly, for any open cover $\mathcal{U}$ of $\textbf{M}$ there is a bijection between the set $\text{Line}_{\mathbb{R}}(\textbf{M}, \mathcal{U})$ of real line bundles trivialised by $\mathcal{U}$ and the $\mathcal{U}$-dependent Čech cohomology $\check{H}^1(\textbf{M}, \mathcal{U}, \mathbb{Z}_2)$. We will now use the direct limit construction of $\check{H}^1(\textbf{M}, \mathbb{Z}_2)$ to prove the following.

\begin{theorem}\label{THM: line bundles for NH manifold}
 There is a bijection between the set $\text{Line}_{\mathbb{R}}(\textbf{M})$ of inequivalent real line bundles over $\textbf{M}$ and the first Čech cohomology group $\check{H}^1(\textbf{M}, \mathbb{Z}_2)$. 
\end{theorem}
\begin{proof}
    For each open cover $\mathcal{U}$ of $\textbf{M}$, denote by $\alpha_{\mathcal{U}}$ the bijection between $\text{Line}_{\mathbb{R}}(\textbf{M},\mathcal{U})$ and $\check{H}^1(\textbf{M},\mathcal{U})$. By construction, each group $\check{H}^1(\textbf{M},\mathcal{U}, \mathbb{Z}_2)$ maps into $\check{H}^1(\textbf{M}, \mathbb{Z}_2)$ via some map $\psi_{\mathcal{U}}$, and this map commutes with any refinement maps. We define a map $\alpha: \text{Line}_{\mathbb{R}}(\textbf{M}, \mathbb{Z}_2)\rightarrow \check{H}^1(\textbf{M}, \mathbb{Z}_2)$ by sending each $\textbf{L}$ to $\psi_{\mathcal{U}}\circ \alpha_{\mathcal{U}}(\textbf{L})$, where $\mathcal{U}$ is some open cover that trivialises $\textbf{L}$. To see that the map $\alpha$ is well-defined, suppose that $\textbf{L}$ is some line bundle that is simultaneously trivialised by two open covers $\mathcal{U}$ and $\mathcal{V}$. We may then select some open cover $\mathcal{W}$ that refines both $\mathcal{U}$ and $\mathcal{V}$ simultaneously. Consider two maps $\lambda_{\mathcal{U}}$ and $\lambda_{\mathcal{V}}$ that encode the refinements into $\mathcal{W}$. By the direct limit construction of $\check{H}^1(\textbf{M}, \mathbb{Z}_2)$, we then have that 
    $$ \psi_{\mathcal{U}} \circ \alpha_{\mathcal{U}}(\textbf{L}) = \psi_{\mathcal{W}} \circ \lambda_{\mathcal{U}} \circ \alpha_{\mathcal{U}} (\textbf{L})= \psi_{\mathcal{W}}\circ \alpha_{\mathcal{W}}(\textbf{L}) = \psi_{\mathcal{W}} \circ \lambda_{\mathcal{V}} \circ \alpha_{\mathcal{V}} (\textbf{L}) = \psi_{\mathcal{V}} \circ \alpha_{\mathcal{V}}(\textbf{L}),     $$
    as required. Finally, we observe that $\alpha$ is bijective since every $\alpha_{\mathcal{U}}$ is.
\end{proof}
Given the results of Sections 3.1 and 3.2, it would be useful if there were formulas to express the order of $\check{H}^1(\textbf{M}, \mathbb{Z}_2)$ in terms of the orders of the groups $\check{H}^1(M_i, \mathbb{Z}_2)$. Unfortunately, a general formula may not exist. However, we may use Theorem \ref{THM: 1st Čech cohomology is fibred product for connected colimits} together with a basic property of fibred products to yield the following.
\begin{theorem}\label{THM: main theorem for line bundles}
Let $\textbf{M}$ be a non-Hausdorff manifold formed from Hausdorff manifolds $M_i$, according to \ref{THM: Summary of topological properties}. Suppose furthermore that: 
\begin{enumerate}
    \item the subspaces $M_{{i_1}\cdots {i_p}}$ are all connected, for all $p \leq n$, and 
    \item each of the descended difference maps $$\tilde{\delta}: \bigoplus_{i_1 < \cdots < i_p} \check{H}^1(M_{i_1 \cdots i_p}, \mathbb{Z}_2) \rightarrow 
\bigoplus_{i_1 < \cdots < i_{p+1}} \check{H}^1(M_{i_1 \cdots i_{p+1}}, \mathbb{Z}_2)$$ is surjective. 
\end{enumerate}
Then the number of inequivalent line bundles on $\textbf{M}$ can be expressed with the formula:
$$ |\check{H}^1(\textbf{M}, \mathbb{Z}_2)| =  \sum_{k \leq n} \sum_{i_1,\cdots i_p \in I} (-1)^{(p+1)} |\check{H}^1( M_{{i_1} \cdots {i_p}}, \mathbb{Z}_2)|.             $$
\end{theorem}

\section{Conclusion}
In this paper we have explored the prospect of fibering vector spaces over a base space which has the structure of a non-Hausdorff manifold. Using the topological criteria outlined in both Section 1.1 and \cite{oconnell2023non}, we saw that bundle structures can be naturally defined, reconstructed, and in some cases classified. Our initial observation was that of Theorem \ref{LEM: non-Haus smooth manifold}, which confirmed that smooth non-Hausdorff manifolds can be constructed by gluing together ordinary smooth manifolds along diffeomorphic open submanifolds. This, once combined with the contravariance of the $C^\infty$ functor, then allowed us to express the algebra of smooth real-valued functions of a non-Hausdorff manifolds as a limit in the abelian category of unital associative algebras. \\

Throughout Section 2, we saw that vector bundles fibred over our non-Hausdorff manifold $\textbf{M}$ can be constructed by taking a particular colimit of Hausdorff bundles. This was proved as Theorem \ref{THM: colimit of bundles}, and it was later shown in Theorem \ref{THM: reconstruction theorem for colimit bundles} that every vector bundle over $\textbf{M}$ can be described in this manner. Using this observation, we then showed that all sections of all bundles over $\textbf{M}$ carry a description in terms of sections on each Hausdorff submanifold. In Theorem \ref{THM: E admits a bundle metric} this piecewise construction of sections was used to conclude that bundle metrics will exist for any non-Hausdorff bundle $\textbf{E}$ fibred over $\textbf{M}$. In particular, Corollary \ref{COR: M admits Riemannian metric} confirmed that Riemannian metrics always exist on any non-Hausdorff manifold $\textbf{M}$ satisfying our particular topological criteria. \\

In Section 3 we studied the Čech cohomology of non-Hausdorff manifolds. In Sections 3.1 and 3.2 we related the Čech cohomology groups $\check{H}^q(\textbf{M})$ to the groups $\check{H}^q(M_{i_1 \cdots i_p})$ in the form of (generalised) Mayer-Vietoris sequences. These sequences will be typically be non-trivial, and therefore the fibre-product structure seen in \ref{THM: fibred product of smooth functions}, \ref{THM: fibred product of sections} and \ref{COR: general Čech cochain fibre product} will not present itself in Čech cohomology. Despite that, we saw in the form of Theorem \ref{THM: 1st Čech cohomology is fibred product for connected colimits} that the first Čech cohomology group $\check{H}^1(\textbf{M})$ \textit{will} be a fibred product, provided that the Hausdorff submanifolds $M_i$ and all their intersections are connected. Finally, in Theorem \ref{THM: main theorem for line bundles} we identified some conditions under which the inequivalent real lines bundles over $\textbf{M}$ can be expressed as an alternating sum formula. \\

We will finish this paper with some speculations regarding future developments. Of particular interest is the relationship between the Čech cohomology groups established here, and other theories such as de Rham cohomology. As an application of Theorem \ref{THM: fibred product of sections}, the differential forms on $\textbf{M}$ satisfy a fibred product structure, and therefore it will be interesting to compute the de Rham cohomology groups (à la \cite{crainic1999remark}) and relate them to the groups $\check{H}^q(\textbf{M}, \mathbb{R})$. These relationships might then be used in conjunction with a sound theory of affine connections to establish a Chern-Weil Theory for non-Hausdorff manifolds.

\subsection*{Acknowledgements}
This paper was made possible by the funding I have received from the Okinawa Institute of Science and Technology. In no particular order, I'd like to thank Yasha Neiman, Slava Lysov, Tim Henke, Andrew Lobb, Chris Chung and Martín Forsberg Conde for the various discussions. 

\printbibliography

\appendix
\section{Appendix}

\begin{lemma}\label{LEM:Appendix inductive colimit}
Let $\textbf{M}$ be a non-Hausdorff manifold built from three manifolds $M_i$. Then $\textbf{M}$ is homeomorphic to the two-step adjunction space $$ (M_1 \cup_{f_{12}} M_2)\cup_{f_{13} \cup f_{23}} M_3.$$ 
\end{lemma}
\begin{proof}
Let $\textbf{N}$ denote the adjunction of $M_1$ and $M_2$ along the open subspace $M_{12}$. We will use $\varphi_1$ and $\varphi_2$ to denote the canonical embeddings of $M_1$ and $M_2$, and brackets $\llbracket \cdot, \cdot \rrbracket$ to denote points in $N$. Consider the subspace $A:= M_{13}\cup M_{23}$, viewed as a subspace of $M_3$. We define a function $f: A \rightarrow \textbf{N}$ by 
$$f(x) = \begin{cases} \llbracket f_{31}(x),1 \rrbracket & \text{if} \ x\in M_{13} \\ \llbracket f_{32}(x),2 \rrbracket & \text{if} \ x\in M_{23}   \end{cases}$$
The function $f$ is well-defined since any element $x$ in the intersection $M_{13}\cap M_{23}$ will lie in $M_{12}$, thus $f_{12}(x) = f_{31} \circ f_{23}(x)$. Diagrammatically, our argument thus far can be depicted as follows, 
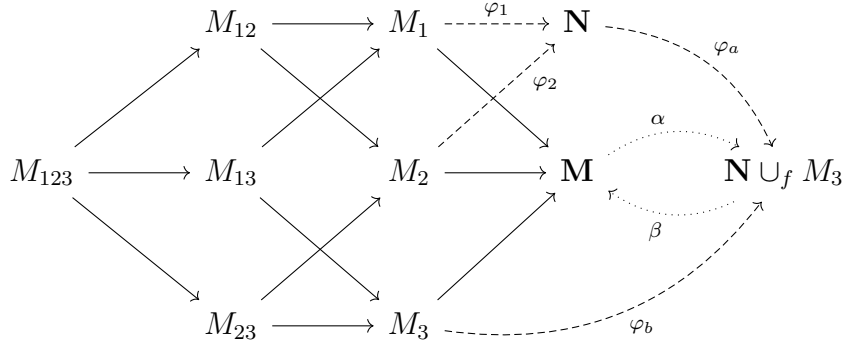
\begin{center}
% https://tikzcd.yichuanshen.de/#N4Igdg9gJgpgziAXAbVABwnAlgFyxMJZARgBoAGAXVJADcBDAGwFcYkQBZAfWGICYAviAGl0mXPkIoyxanSat23XgGYhIsdjwEiZPnIYs2iTjz5rhokBi2SifCgYXHTxS5ok6UD2TUOKTbj53a3FtKWQHfT9nJS4VEJtPCJVHGKN2AB1MnBgADxwAIwAzYAA5dSsk8KIAFlJfeQyTbNyCkvKBbIBjZjQeVqxGWGBigQFuBI1Q2y9kVMb-F1b8otKOSo8alHIGp2bTXnN1ORgoAHN4IlBigCcIAFskMhAcCCQ1KzvH55o3pAArNNvk9ELtXu9ELVgfdQQ4IUhoV9YUhwf9EJ8bijEPD0UDkT9EAC-pCAOzpAIgbIMW5oAAWWC4WBANEY9EKMEYAAUwnYTLcsOc6TgWSAOWAoEgALQqcGMLBgFxQehwOlnEIgpAANhJSHJTUp1PotIZTNF4slGLlCqVKrVkphhPqCMQOpA8sV7GVqvVFOWmQABNkg4Hg2HQxGQ1Hg9HI+HY1GE-HkxHuqbgqz2ZyebMpO6YMURY7QQAOXWIfUepUQHC5SU0NX0S1gZiMRh+rKZJj0+jmmASj7kYsfcsl4dE0fj53oscE0EATnLQ7nSEXLuI47X6MEK4x5bdVa9dt9Bv9acZbnH+pnNAtg9ZNq9NbroqWncKxuA2W7dPoJwEQA
\begin{tikzcd}[row sep=3.2em, column sep=3.3em]
                                        & M_{12} \arrow[rd] \arrow[r]  & M_1 \arrow[rd] \arrow[r, "\varphi_1", dashed]                                                       & \textbf{N} \arrow[rd, "\varphi_a", dashed, bend left] &                                                                            \\
M_{123} \arrow[ru] \arrow[r] \arrow[rd] & M_{13} \arrow[ru] \arrow[rd] & M_2 \arrow[ru, "\varphi_2"', near end, dashed] \arrow[r] & \textbf{M} \arrow[r, "\alpha", dotted, bend left]     & \textbf{N}\cup_{f}M_3 \arrow[l, "\beta", dotted, bend left] \\
                                        & M_{23} \arrow[ru] \arrow[r]  & M_3 \arrow[rru, "\varphi_b
                                        "', dashed, bend right] \arrow[ru]                                     &                                                       &                                                                           
\end{tikzcd}
\end{center}

where here the maps $\chi_i$ and $\varphi_i$ are all the canonical embedding maps of the various (binary) adjunction spaces. Observe that all three $M_i$ naturally embed into the space $\textbf{N}\cup_{f} M_3$. Indeed: for $M_1$ and $M_2$ we may use the compositions $\varphi_a \circ \chi_i$, and for $M_3$ we may simply use $\varphi_b$. Moreover, by construction these maps commute with all of the $\iota_{ij}$ and $f_{ij}$ maps. As such, we may employ the universal property \ref{LEM: universal property of adjunction space} of $\textbf{M}$ to conclude that there exists a unique continuous map $\alpha: \textbf{M} \rightarrow \textbf{N}\cup_{f}M_3$. Explicitly, this map will act as:
$$  \alpha([x,i]) = \begin{cases} \big[[x,1],a \big] & \textrm{if} \ x\in M_1 \\ 
\big[[x,2],a \big] & \textrm{if} \ x\in M_2 \\
\big[[x,3],b \big] & \textrm{if} \ x\in M_3 \end{cases}.        $$
We can use the maps $\eta: \textbf{N} \rightarrow \textbf{M}$ and $\phi_3: M_3 \rightarrow \textbf{M}$ together with the universal property of the binary adjunction space $\textbf{N}\cup_{f}M_3$ to conclude that there is a unique map $\tilde{\alpha}:\textbf{N}\cup_{f}M_3 \rightarrow \textbf{M}$. A simple exercise confirms that the maps $\alpha$ and $\tilde{\alpha}$ are inverses of each other. \end{proof}

\begin{theorem}\label{THM: Fibred product is Inductive Product}
Suppose that the diagram $\mathcal{F}$ appears in an Abelian category in inverse form. Then the general fibred product $\prod_{\mathcal{F}} A_i$ is isomorphic to an inductive limit.
\end{theorem}   
\begin{proof}
    We proceed via induction. Suppose first that the indexing set $I$ has size $3$. We will show that $\textbf{A}$ is isomorphic to the two-step construction $(A_1 \times_{A_{12}} A_2)\times_B A_3$, where $B:=A_{13}\times_{A_{123}} A_{23}$. The following diagram depicts the schematics of the result.   
\begin{center}
\begin{tikzcd}[row sep=1.4em, column sep=1.4em]
        &  & A_{12} \arrow[lldd]   &  & A_1 \arrow[ll] \arrow[lldd]   &  & A_1 \times_{A_{12}} A_2 \arrow[ll, dashed] \arrow[lldd, dashed]              &                                                                                                                                      \\
        &  &                       &  &                               &  &                                                                              &                                                                                                                                      \\
A_{123} &  & A_{13} \arrow[ll]     &  & A_2 \arrow[lluu] \arrow[lldd] &  & \textbf{A} \arrow[lldd] \arrow[ll] \arrow[lluu] \arrow[r, dotted, bend left] & (A_1 \times_{A_{12}} A_2) \times_B A_3 \arrow[luu, dashed, bend right] \arrow[llldd, dashed, bend left] \arrow[l, dotted, bend left] \\
        &  & B \arrow[u] \arrow[d] &  &                               &  &                                                                              &                                                                                                                                      \\
        &  & A_{23} \arrow[lluu]   &  & A_3 \arrow[lluu] \arrow[ll]   &  &                                                                              &                                                                                                                                     
\end{tikzcd}

\end{center}
    We may use the universal property of $B$ to construct two maps $\alpha: A_1 \times_{A_{12}} A_2 \rightarrow B$ and $\beta: A_3 \rightarrow B$. Firstly, $\alpha$ is defined by applying the universal property to the two maps that following the sequence  $A_1\times_{A_{12}} A_2 \rightarrow A_i \rightarrow A_{i3} \rightarrow A_{123}$. The map $\beta$ is defined by applying the universal property of $B$ to the two maps following the sequence $A_3 \rightarrow A_{i3} \rightarrow A_{123}$. The two-step fibred product may then be defined as $$   (A_1 \times_{A_{12}} A_2)\times_B A_3 := \{  (a,b) \in A_1 \times_{A_{12}} A_2) \times A_3 \ | \ \alpha(a) = \beta(b)\}.     $$
    We may then use the same argument as that of the previous Lemma and use the universal properties of both $\textbf{A}$ and $(A_1 \times_{A_{12}} A_2)\times_B A_3$ to construct the desired isomorphism. For the inductive argument, we can proceed in essentially the same manner as that of Theorem \ref{THM: inductive colimit}, except that this time we reverse all the arrows and take limits instead.         
\end{proof}

\subsection{The Cover-Independent Mayer-Vietoris Sequence}
In Section 3.1 we determined the Mayer-Vietoris sequence for Čech cohomology relative to an open cover $\mathcal{U}$. We will now briefly justify a similar relationship between the Čech cohomologies that arise via direct limits. Suppose that $\mathcal{V}=\{V_\beta\}_{\beta \in B}$ is a refinement of the open cover $\mathcal{U} = \{U_{\alpha} \}_{\alpha \in A}$. This refinement can be formalised using a map $\lambda:B\rightarrow A$ defined by the property $V_\beta \subseteq U_{\lambda(\beta)}$ for all $\beta$ in $B$. There may be many such maps for the same refinement, however the following argument is independent of the particular choice of $\lambda$.

\begin{lemma}
    The refinement map $\lambda$ commutes with both $\Phi^*$ and $\iota^*_{12}-f_{12}^*$. 
\end{lemma}
\begin{proof}
    Let $\check{f}$ be a cochain in $\check{C}^q(\textbf{M}, \mathcal{V})$. For any function $f_{\beta_0 \cdots \beta_q}$ in $\check{f}$, we have $$  \lambda(\phi_i^*\check{f})_{\beta_0 \cdots \beta_q} = \lambda (f_{\beta_0 \cdots \beta_q} \circ \phi_i) = f_{\lambda(\beta_0)\cdots\lambda(\beta_q)} \circ \phi_i = \phi_i^*(\lambda \check{f}).       $$ 
    Then $$ \lambda \circ \Phi^* \check{f} = (\lambda \circ \phi_1^* \check{f}, \lambda \circ \phi_2^* \check{f}) =  (\phi_1^* \circ \lambda \check{f}, \phi_2^* \circ \lambda\check{f}) = \Phi^* \circ \lambda(\check{f}).  $$
    The case for $\iota_{12}^*$ and $f_{12}^*$ are similar. 
\end{proof}
A standard computation confirms that the refinement map $\lambda$ commutes with the Čech differential \cite{bott1982differential}. This, together with the previous result, allows us to view $\lambda$ as a chain map. We may then apply the Snake Lemma, together with the descended map $\lambda$, to obtain the following commutative relationship between the two cover-dependent Mayer-Vietoris sequences. 
\begin{center}
\adjustbox{scale=0.87}{
   % https://tikzcd.yichuanshen.de/#N4Igdg9gJgpgziAXAbVABwnAlgFyxMJZARgBoAGAXVJADcBDAGwFcYkQAdDgYwAsZuAa2AAJAL4A9AI4AKLjhgAPHACMAZsACyY0gAIuKiIyhwAngFtDjYF3P0cvbk2ABVMWICUIHeky58hChkxNR0TKzsXHwCwuLSchwKyupaOvocViYWVjYcdg5O1gBq7l4+IBjYeAREAEykITQMLGyInDz8QqKSspoA+mTp+Y7OJRLEHlwQaCxw6dFdcb199UP2I8WStWWkvlUBdRShzRFtUZ2xPTL9g7brha6SE1MzzHPnMd3x-at3Bc5uCTbby7Cp+aqBZAAZgax3CrXaC0u3z6wGItTSfw2wDGaIxOz2-hqKBhVCa8MiHU+S2uqPRmLy9wBkjxnhBhIhRHIR3JLUp3CgEBwCHKlSJkO5jTCfLOPEFwvZYP2xOQABZYbzToj5SLQWLOSh1WTpVqojrvKEYFAAObwIigNQAJwg5iQ3JAOAgSBhJoRXAACrwsBIAFSKp0u700T1IACsmr9HHwOHodIxod0AFpdGo05Iw+UI67EGQPV7EPVfZTA8GwzRGPQVDBGP7wQc2o6sNbeDhw87i5WY4h1VXZcnU6zQ5nc5O6yAG02W23iSBO93e4X+0hS0P3SdEw3LFB6H3IxXo+WffvKYeVMfT8WR0P46P2rf75uzwA2C9uhOU2BGBTUMHyQAB2X8S3-WVAOAucF2bVtlUCVcux7UDh0ggAOaD2lg+gQPrRtEOXFC13Qz9ixfIcAE5cK4fCQLESgxCAA
\begin{tikzcd}[row sep=3.9em]
\cdots \arrow[r, "\delta^*"]  & {\check{H}^q(\textbf{M}, \mathcal{U})} \arrow[r, "\Phi^*"]                       & {\check{H}^q(M_1, \mathcal{U}^1)\oplus \check{H}^q(M_2, \mathcal{U}^2)} \arrow[r, "\iota_{12}^* - f_{12}^*"]                     & {\check{H}^q(M_{12}, \mathcal{U}^{12})} \arrow[r, "\delta^*"]                       & \cdots \\
\cdots \arrow[r, "\delta^*"'] & {\check{H}^q(\textbf{M}, \boldsymbol{\mathcal{V}})} \arrow[r, "\Phi^*"'] \arrow[u, "\lambda"] & {\check{H}^q(M_1, \mathcal{V}^1)\oplus \check{H}^q(M_2, \mathcal{V}^2)} \arrow[r, "\iota_{12}^*-f_{12}^*"'] \arrow[u, "\lambda"] & {\check{H}^q(M_{12}, \mathcal{V}^{12})} \arrow[u, "\lambda"] \arrow[r, "\delta^*"'] & \cdots
\end{tikzcd}}
\end{center}

Note that the map $\lambda$ commutes with the connecting homomorphisms $\delta^*$ due to the naturality of the Snake Lemma. The cover-independent Čech cohomologies are direct limits that satisfy certain universal properties. Since the refinement maps $\lambda$ form commutative diagrams pictured above, we may invoke these universal properties of directed limits (by using the right choice of compositions) in order to construct a unique Mayer-Vietoris sequence 
$$     \cdots \xrightarrow{\delta^*} \check{H}^q(\textbf{M}) \xrightarrow{\Phi^*} \check{H}^q(M_1) \oplus \check{H}^q(M_2) \xrightarrow{\iota_{12}^* - f_{12}^*} \check{H}^q(M_{12}) \xrightarrow{\delta^*} \check{H}^{q+1}(\textbf{M}) \rightarrow \cdots                 $$
as required.

\end{document}